\documentclass[reqno,11pt]{amsart}
\usepackage[margin=1in]{geometry}
\usepackage{amsmath,amssymb,amsthm,amsfonts,mathrsfs}
\usepackage{color}
\usepackage{cite}
\usepackage{enumitem}
\usepackage[colorlinks=true,linkcolor=blue,citecolor=blue,urlcolor=blue]{hyperref}
\allowdisplaybreaks
\newtheorem{theorem}{Theorem}[section]
\newtheorem{lemma}[theorem]{Lemma}

\newtheorem{corollary}[theorem]{Corollary}

\theoremstyle{remark}
\newtheorem {remark}{{\bf Remark}}[section]

\theoremstyle{plain} \numberwithin {equation}{section}

\newcommand{\R}{\mathbb{R}}

\newcommand{\Z}{\mathbb{Z}}
\newcommand{\p}{\partial}
\newcommand{\eps}{\varepsilon}

\DeclareMathOperator{\Rea}{Re}

\DeclareMathOperator{\spec}{\sigma}
\DeclareMathOperator{\diag}{diag}

\newcommand{\Lan}{\mathcal{L}_{a,\nu}}

\title[Kolmogorov flow with horizontal dissipation]{Sharp linear stability and the absence of enhanced dissipation for Kolmogorov flow in the 2D Navier--Stokes equations with horizontal dissipation}

\author{Wanrong Yang}
\address{School of Mathematics and Information Sciences, North Minzu University,
	Yinchuan, Ningxia 750021, PR China}
\email{yangwanrong1618@163.com}

\author{Jiahong Wu}
\address{Department of Mathematics, University of Notre Dame,
	Notre Dame, IN 46556, USA}
\email{jwu29@nd.edu}

\author{Xiaoping Zhai}
\address{School of Mathematics and Statistics, Guangdong University of Technology,
	Guangzhou 510520, PR China}
\email{pingxiaozhai@163.com (Corresponding author)}
\date{\today}
\subjclass[2020]{35Q35; 35B35; 76D05; 76E05; 35B40}
\keywords{Kolmogorov flow; anisotropic Navier--Stokes equations; horizontal
	dissipation; partial dissipation; linear stability; enhanced dissipation;
	shear flow; spectral instability}

\begin{document}

\begin{abstract}
We study the linearized dynamics of the Kolmogorov flow $U^{(0)}=(a\sin x_2,0)$ for the two-dimensional incompressible Navier--Stokes equations with horizontal dissipation $\nu\partial_1^2$. Unlike the fully dissipative case, this anisotropic system admits $U^{(0)}$ as an exact unforced steady state. On each horizontal Fourier mode $k$, the dissipation reduces to the scalar $-\nu k^2$ and commutes with the advection, so that the linearized semigroup factors exactly into $e^{-\nu k^2t}$ times the inviscid Euler group. For $|k|>1$, we prove two-sided bounds showing that the decay rate is exactly $\nu k^2$, uniformly in the shear amplitude $a$; hence no enhanced dissipation occurs. This reflects a fundamental mismatch: the shear transfers enstrophy to high vertical frequencies, which horizontal dissipation does not detect. At the critical modes $|k|=1$, the inviscid group grows like $\sqrt{2at}$, producing a transient amplification of size $(a/e\nu)^{1/2}$ before decay at the rate $\nu$. The horizontally independent modes form an infinite-dimensional undamped kernel. For $0<|k|<1$, the viscous spectrum is an exact translate of the inviscid one, and the mode is linearly unstable if and only if $a\Lambda(k)>\nu k^2$. In particular, when $L>2\pi$, the flow is linearly unstable for all sufficiently large $a$.
\end{abstract}

\maketitle
\section{Introduction}

\subsection{The model}

We consider the two-dimensional incompressible Navier--Stokes equations
with dissipation acting only in the horizontal direction:
\begin{equation}\label{ans}
\begin{cases}
\partial_t U+U\cdot\nabla U=-\nabla P+\nu\partial_1^2U,\\
\nabla\cdot U=0,
\end{cases}
\qquad
(x_1,x_2)\in\mathbb T_L\times\mathbb T_{2\pi},
\quad t>0,
\end{equation}
where
\begin{align*}
\mathbb T_L=\mathbb R/L\mathbb Z,
\qquad
\mathbb T_{2\pi}=\mathbb R/2\pi\mathbb Z,
\qquad
\nu>0.
\end{align*}
More generally, \eqref{ans} is the case $b=(1,0)$ of the directionally
dissipative system
\begin{align}\label{nsb}
\partial_t U+U\cdot\nabla U
=
-\nabla P+\nu(b\cdot\nabla)^2U,
\qquad
\nabla\cdot U=0,
\end{align}
where $b\in\mathbb R^2$ is a fixed unit vector. Equations with only partial or directional dissipation arise naturally in geophysical and atmospheric modeling, where the effective eddy viscosities in different directions may differ by orders of magnitude, and they have attracted considerable mathematical attention in recent years; see, e.g., \cite{CheminDesjardinsGallagherGrenier,CaoWu,DoeringWuZhaoZheng,DongWuXuZhu} and the references therein. Partial dissipation has also been studied in the stability theory of shear flows; see, e.g., \cite{DengWuZhang,LiangWuZhai}. From the analytic viewpoint, the interest of \eqref{ans} lies in the fact that the dissipation controls only one family of frequencies: on the Fourier side, the semigroup generated by $\nu\partial_1^2$ has multiplier $e^{-\nu t\xi_1^2}$, which provides no decay and no smoothing whatsoever on the set $\{\xi_1=0\}$ and, more generally, provides decay that is uniform in the vertical frequency $\xi_2$.

The Kolmogorov flow is a fundamental model for hydrodynamic stability,
metastability, and transition to turbulence. Its classical spectral and
nonlinear stability theory was initiated by Meshalkin and Sinai
\cite{MeshalkinSinai} and Yudovich \cite{Yudovich}. {On the square torus, the forced Kolmogorov flow is globally stable for every viscosity \cite{Marchioro}, whereas on elongated tori it loses stability at long wavelengths; see \cite{MeshalkinSinai,Yudovich} and, for a recent rigorous treatment of the long-wave instability of general periodic shear flows, \cite{ColomboDolceMontaltoVentura}.} For the standard
two-dimensional Navier--Stokes equations with full Laplacian viscosity,
Ibrahim, Maekawa, and Masmoudi \cite{IbrahimMaekawaMasmoudi} established,
by pseudospectral analysis, the sharp enhanced-dissipation rate
$
e^{-c\sqrt{a\nu}\,t}
$
for the linearized dynamics after the neutral modes are removed{; see also \cite{BeckWayne,LinXu,WeiZhang,WeiZhangZhao} for related linear results and \cite{BeekieChenJia} for more general periodic shear flows}.
Li, Wei, and Zhang \cite{LiWeiZhang} extended the pseudospectral and
wave-operator methods to the three-dimensional Kolmogorov flow and,
when the background amplitude is normalized independently of $\nu$,
obtained an $H^2$ nonlinear stability threshold of order
$\nu^{7/4}$. More recently, Chen, Jia, Wei, and Zhang
\cite{ChenJiaWeiZhang} proved nonlinear asymptotic stability in two
dimensions for $H^3$ perturbations of size $o(\nu^{1/3})$, using a
combination of enhanced dissipation, inviscid damping, and vorticity
depletion. {In a complementary direction, Coti Zelati, Elgindi, and Widmayer \cite{CotiZelatiElgindiWidmayerARMA} constructed stationary structures near the Kolmogorov flow for the Euler equations and showed that nonlinear enhanced-dissipation results of the above type cannot hold near the Kolmogorov flow on the square torus.}

These results rely crucially on the fact that shear-induced mixing
transfers vorticity toward large transverse frequencies, where the full
Laplacian produces increasingly strong dissipation{; see \cite{ConstantinKiselevRyzhikZlatos,BedrossianCotiZelati} for this general mechanism}. This mechanism is absent a priori in \eqref{ans}: on each fixed horizontal Fourier mode, the operator $\nu\partial_1^2$ reduces to a scalar multiplication and therefore cannot resolve the vertical frequencies induced by the shear. This leads to the central
question of the present paper: whether mixing by the Kolmogorov flow can
still accelerate viscous decay when dissipation acts only along the flow
direction. Our purpose is to give a sharp mode-by-mode description of
the resulting linearized dynamics, including the stable, critical,
neutral, and long-wave unstable regimes, and to determine precisely the
dependence on the shear amplitude $a$.

\subsection{Constant backgrounds and Galilean reduction}
Before introducing the shear flow, it is worth recording why constant background flows are not an interesting testing ground for the question of large-amplitude behavior. Both the fully dissipative Navier--Stokes equations and the anisotropic system \eqref{nsb} admit the constant steady states $U^{(0)}=a e$, $P^{(0)}=0$, for any constant unit vector $e$ and amplitude $a>0$. The equations for the perturbation $u=U-U^{(0)}$ contain the transport term $a(e\cdot\nabla)u$. However, the Galilean change of variables
\begin{align*}
v(x,t) := u(x+aet,\,t)
\end{align*}
removes this term exactly, while leaving the nonlinearity, the incompressibility constraint, and the (constant-coefficient) dissipation invariant. Since translations act as isometries on every translation-invariant function space, every stability threshold and every decay rate for the perturbation problem is uniform in $a$: on the linear level the solution is $u(x,t)=(e^{\nu t\p_1^2}u_0)(x-aet)$, whose norms are independent of $a$, and on the nonlinear level all estimates for the Duhamel formulation involve the oscillatory factor $e^{-ia(t-s)(e\cdot\xi)}$ only through its modulus, which is one. A large constant background thus neither enhances dissipation nor destabilizes. Genuinely amplitude-dependent phenomena require a background that breaks Galilean invariance, the simplest examples being shear flows. This observation motivates the choice of background made in this paper.

\subsection{Kolmogorov flow as an exact steady state of the anisotropic system}

Let
\begin{align}\label{Kolmo}
U^{(0)}(x)
=
\bigl(a\sin x_2,0\bigr),
\qquad
P^{(0)}=0,
\qquad
a>0.
\end{align}
In the fully dissipative Navier--Stokes equations,
\begin{align*}
\Delta U^{(0)}=-a\sin x_2\,e_1,
\end{align*}
so \eqref{Kolmo} must be sustained by the forcing
\begin{align*}
f=\nu a\sin x_2\,e_1;
\end{align*}
see \cite{MeshalkinSinai,Yudovich}. By contrast, for the horizontally
dissipative system \eqref{ans},
\begin{align*}
U^{(0)}\cdot\nabla U^{(0)}
=
a\sin x_2\,\partial_1U^{(0)}
=
0,
\qquad
\partial_1^2U^{(0)}=0.
\end{align*}
Hence \eqref{Kolmo} is an exact unforced equilibrium of \eqref{ans} for
all $a>0$ and $\nu>0$.  The stability problem studied below is therefore the stability of a genuine, unforced equilibrium.

\vskip .1in
\subsection{The linearized problem}
Set
\begin{align*}
u:=U-U^{(0)},
\qquad
p:=P-P^{(0)}.
\end{align*}
Neglecting the quadratic term $u\cdot\nabla u$, we obtain
\begin{align}\label{linvel}
\partial_t u
+a\sin x_2\,\partial_1u
+a\cos x_2\,u_2e_1
=
-\nabla p+\nu\partial_1^2u,
\qquad
\nabla\cdot u=0.
\end{align}

Introduce the perturbation vorticity and stream function by
\begin{align*}
\omega:=\partial_1u_2-\partial_2u_1,
\qquad
u=\nabla^\perp\psi,
\qquad
\psi=\Delta^{-1}\omega,
\end{align*}
where $\Delta^{-1}$ is understood on the zero-mean subspace.

Taking the curl of \eqref{linvel} therefore yields
\begin{align}\label{linvort}
\partial_t\omega
+
a\sin x_2\,\partial_1
\bigl(1+\Delta^{-1}\bigr)\omega
=
\nu\partial_1^2\omega.
\end{align}
The Fourier multiplier $1+\Delta^{-1}$ is the characteristic
Kolmogorov-flow structure. Its associated quadratic form ceases to be
positive definite at sufficiently low frequencies, leading to the long-wave
instability discussed below.

Expanding in horizontal Fourier modes,
\begin{align*}
\omega(x_1,x_2,t)
=
\sum_{k\in(2\pi/L)\mathbb Z}
e^{ikx_1}\widehat\omega_k(x_2,t),
\end{align*}
we obtain, for $k\neq0$,
\begin{align}\label{modek}
\partial_t\widehat\omega_k
+ika\sin x_2
\left[
1-(k^2-\partial_2^2)^{-1}
\right]\widehat\omega_k
=
-\nu k^2\widehat\omega_k.
\end{align}
The mode $k=0$ is treated separately in
Theorem~\ref{thm:k0}.

For fixed $k\neq0$, write
\begin{align*}
\widehat{\omega}_k(x_2,t)
=
\sum_{n\in\mathbb Z}c_n(t)e^{inx_2}
\end{align*}
and define
\begin{align}\label{lambdadef}
\lambda_n=\lambda_n^{(k)}
:=
1-\frac{1}{k^2+n^2}.
\end{align}
Comparing the coefficients of $e^{inx_2}$ in \eqref{modek}, we
obtain the tridiagonal system
\begin{align}\label{tridiag}
\frac{d}{dt}c_n(t)
=
-\nu k^2c_n(t)
-\frac{ka}{2}
\left(
\lambda_{n-1}^{(k)}c_{n-1}(t)
-
\lambda_{n+1}^{(k)}c_{n+1}(t)
\right),
\qquad n\in\mathbb Z.
\end{align}
A detailed derivation of \eqref{tridiag} is given in
Section~\ref{sec:reduction}.

Writing
\begin{align*}
c(t)=(c_n(t))_{n\in\mathbb Z},
\end{align*}
define, for $k\neq0$,
\begin{align}\label{aknu}
\mathcal A_{k,\nu}
:=
-\nu k^2\mathcal{I}+a\mathcal{M}_k,
\end{align}
where
\begin{align}\label{Mk}
(\mathcal{M}_k\,c)_n
:=
-\frac{k}{2}
\left(
\lambda_{n-1}^{(k)}c_{n-1}
-
\lambda_{n+1}^{(k)}c_{n+1}
\right),
\qquad n\in\mathbb Z.
\end{align}
Then \eqref{tridiag} takes the operator form
\begin{align*}
\frac{d}{dt}c(t)
=
\mathcal A_{k,\nu} \,c(t).
\end{align*}
Moreover,
\begin{align*}
\|\mathcal{M}_k\|_{\mathcal B(\ell^2)}
\leq
|k|
\sup_{n\in\mathbb Z}
\left|\lambda_n^{(k)}\right|
=
|k|
\max\left\{
1,\frac1{k^2}-1
\right\}
<\infty.
\end{align*}
Thus $\mathcal{M}_k$ and $\mathcal A_{k,\nu}$ are bounded operators on
$\ell^2(\mathbb Z)$.

System~\eqref{tridiag} is the anisotropic counterpart of the
Meshalkin--Sinai system \cite{MeshalkinSinai}. Its decisive structural
feature is that, on each fixed horizontal Fourier mode, the
dissipative term is the scalar operator $-\nu k^2\mathcal{I},$ independent of the vertical Fourier index. Consequently, horizontal dissipation does not detect the vertical-frequency transfer generated
by the shear.

Indeed, the substitution
\begin{align}\label{substitution}
c_n(t)
=
e^{-\nu k^2t}d_n(t)
\end{align}
reduces \eqref{tridiag} to the normalized inviscid system
\begin{equation*}%
\frac{d}{dt}d_n(t)
=
-\frac{ka}{2}
\left(
\lambda_{n-1}^{(k)}d_{n-1}(t)
-
\lambda_{n+1}^{(k)}d_{n+1}(t)
\right),
\qquad n\in\mathbb Z.
\end{equation*}
Equivalently,
\begin{align}\label{factorization-intro}
e^{t\mathcal A_{k,\nu}}
=
e^{-\nu k^2t}e^{at\mathcal{M}_k},
\qquad t\geq0.
\end{align}
Thus the viscous damping separates exactly from the normalized
inviscid Kolmogorov dynamics.

\subsection{Main results}

We formulate the modewise estimates in the vorticity space
$\ell^2(\mathbb Z)$, equivalently in
$L^2(\mathbb T_{2\pi})$ by Parseval's identity. The operators
 $\mathcal A_{k,\nu}$ and $\mathcal{M}_k$ are defined in
\eqref{aknu} and \eqref{Mk}, respectively.
Throughout the paper, the absence of enhanced dissipation refers to
this fixed-horizontal-mode vorticity norm: no uniform estimate with an
exponential rate strictly larger than $\nu k^2$ can hold on the
$k$-th mode.

The behavior depends sharply on the horizontal wavenumber. The four
regimes are
\begin{align*}
|k|>1,\qquad |k|=1,\qquad k=0,\qquad 0<|k|<1.
\end{align*}

Our first result gives a sharp two-sided estimate in the coercive range.

\begin{theorem}[Sharp decay and absence of enhanced dissipation for $\lvert k\rvert>1$]
\label{thm:A}
Fix a nonzero horizontal Fourier mode $k$ with $\lvert k\rvert>1$, and let
$c(t)=(c_n(t))_{n\in\mathbb Z}$ be a solution of
\eqref{tridiag}. Define
\begin{align*}
E_k[c](t)
:=
e^{2\nu k^2t}
\sum_{n\in\mathbb Z}
\lambda_n^{(k)}\lvert c_n(t)\rvert^2.
\end{align*}
Then, for all $t\ge0$, there holds
\begin{align*}
E_k[c](t)=E_k[c](0).
\end{align*}
Moreover,
\begin{align}\label{twosided}
\left(1-\frac{1}{k^2}\right)^{1/2}
e^{-\nu k^2t}\|c(0)\|_{\ell^2}
\leq
\|c(t)\|_{\ell^2}
\leq
\left(1-\frac{1}{k^2}\right)^{-1/2}
e^{-\nu k^2t}\|c(0)\|_{\ell^2},
\qquad t\geq0.
\end{align}

The constants in \eqref{twosided} are independent of the amplitude $a$,
the dissipation coefficient $\nu$, and time $t$ (but depend on $k$).
Consequently, the $k$-th horizontal mode has the sharp exponential decay
rate $\nu k^2$. In particular, no decay rate strictly faster than
$e^{-\nu k^2t}$ can hold uniformly for all initial data, regardless of the
size of $a$; hence no enhanced dissipation occurs on this mode.
\end{theorem}

{Theorem~\ref{thm:A}} should be contrasted with the fully dissipative
Kolmogorov problem, where, after removing the neutral modes, the
linearized semigroup decays at the enhanced rate
$
e^{-c\sqrt{a\nu}\,t}.
$
For fixed shear amplitude $a$, this corresponds to the time scale
$O(\nu^{-1/2})$, rather than the bare diffusive scale
$O(\nu^{-1})$,
 see {\cite{BeckWayne,LinXu,WeiZhangZhao,IbrahimMaekawaMasmoudi,ChenJiaWeiZhang}} and, for the Couette flow, \cite{BedrossianMasmoudi,BedrossianVicolWang}. The lower bound in \eqref{twosided} shows that in the anisotropic problem the mixing produced by the background shear, no matter how strong, cannot accelerate the decay by more than a fixed constant factor. We emphasize that the constants in \eqref{twosided} degenerate as $|k|\downarrow1$; the marginal case is treated next.

\vskip .1in
At the critical wavenumber $\lvert k\rvert=1$, the coercive energy
degenerates and the bounded inviscid dynamics develops algebraic
transient growth. The next theorem gives the precise growth constant,
the corresponding viscous amplification, and the long-time decay rate.

\begin{theorem}[The marginal wavenumber $|k|=1$]\label{thm:k1}
Let $|k|=1$ and $a,\nu>0$.  Then the inviscid group satisfies the sharp asymptotic law
\begin{align}
\label{critical-inviscid-asymptotics}
\bigl\|e^{a t \mathcal{M}_k}\bigr\|_{\mathcal B(\ell^2)}
=
\sqrt{2at}\,(1+o(1)),
\qquad t\to\infty.
\end{align}
Moreover, there exists an absolute constant $C>0$ such that
\begin{align}
\label{critical-global-bound}
\bigl\|e^{t\mathcal A_{k,\nu}}\bigr\|_{\mathcal B(\ell^2)}
\le
C(1+\sqrt{at})e^{-\nu t},
\qquad t\ge0.
\end{align}
Consequently,
\begin{align}
\label{critical-viscous-asymptotics}
\sup_{t\ge0}
\bigl\|e^{t\mathcal A_{k,\nu}}\bigr\|_{\mathcal B(\ell^2)}
=
\left(\frac{a}{e\nu}\right)^{1/2}(1+o(1)),
\qquad \frac{a}{\nu}\to\infty.
\end{align}
The time $t=(2\nu)^{-1}$ is asymptotically maximizing, in the sense that
\begin{align}
\label{critical-maximizing-time}
\left\|e^{(2\nu)^{-1}\mathcal A_{k,\nu}}\right\|_{\mathcal B(\ell^2)}
=
\left(\frac{a}{e\nu}\right)^{1/2}(1+o(1)),
\qquad \frac{a}{\nu}\to\infty.
\end{align}
In particular,
\begin{align}\label{critical-two-sided}
\sup_{t\ge0}
\bigl\|e^{t\mathcal A_{k,\nu}}\bigr\|_{\mathcal B(\ell^2)}
\asymp
1+\sqrt{\frac{a}{\nu}},
\end{align}
with absolute implicit constants.

If $c(t)=e^{t\mathcal A_{k,\nu}}c(0)$, then the nonzero vertical modes form an
autonomous subsystem and satisfy
\begin{align}
\label{critical-conservation}
e^{2\nu t}
\sum_{n\neq0}\lambda_n|c_n(t)|^2
=
\sum_{n\neq0}\lambda_n|c_n(0)|^2,
\qquad t\ge0.
\end{align}
Thus the nonzero-mode subsystem remains uniformly bounded in the
weighted norm, while the unbounded algebraic growth of the inviscid
semigroup is generated by the forcing of the one-dimensional
$c_0$-component.

Finally, if $e_0$ denotes the standard basis vector supported at $n=0$,
then
\begin{align*}
\mathcal{M}_ke_0=0,
\qquad
e^{t\mathcal A_{k,\nu}}e_0=e^{-\nu t}e_0.
\end{align*}
Together with \eqref{critical-global-bound}, this yields the exact
asymptotic exponential rate
\begin{align*}
\lim_{t\to\infty}
\frac{1}{t}
\log
\bigl\|e^{t\mathcal A_{k,\nu}}\bigr\|_{\mathcal B(\ell^2)}
=
-\nu.
\end{align*}
\end{theorem}

\vskip .1in
The horizontally independent mode $k=0$ must be treated separately.
On this mode, the factors $\partial_1$ and $\partial_1^2$ eliminate
both the linear transport and the horizontal dissipation. The resulting
kernel corresponds precisely, under the Biot--Savart correspondence
with fixed mean velocity, to mean-zero shear perturbations.

\begin{theorem}[The undamped kernel $k=0$]\label{thm:k0}
Let
\begin{align*}
\mathcal X_0
:=
\left\{
\omega\in L^2(\mathbb T_L\times\mathbb T_{2\pi}):
\int_{\mathbb T_L\times\mathbb T_{2\pi}}
\omega(x)\,dx=0
\right\}
\end{align*}
be the physical zero-mean vorticity space. Define
\begin{align*}
\Lan \,\omega
:=
-a\sin x_2\,\partial_1(1+\Delta^{-1})\omega
+\nu\partial_1^2\omega,
\end{align*}
where $\Delta^{-1}$ is the inverse Laplacian on the zero-mean
subspace. Thus \eqref{linvort} takes the form
\begin{align*}
\partial_t\omega=\Lan \,\omega.
\end{align*}
We fix the spatial mean of the perturbation velocity and set
\begin{align*}
\mathcal K_0
:=
\left\{
\omega(x_1,x_2)=g(x_2):
g\in L^2(\mathbb T_{2\pi}),\
\int_0^{2\pi}g(x_2)\,dx_2=0
\right\}.
\end{align*}
Then
\begin{align*}
\Lan \big|_{\mathcal K_0}=0,
\qquad
e^{t\Lan }\omega=\omega
\quad
\text{for every }\omega\in\mathcal K_0\text{ and }t\ge0.
\end{align*}
Under the zero-mean Biot--Savart correspondence, $\mathcal K_0$ is
identified with the space of mean-zero shear perturbations
\begin{align*}
\mathcal S_0
:=
\left\{
u=(f(x_2),0):
f\in H^1(\mathbb T_{2\pi}),\
\int_0^{2\pi}f(x_2)\,dx_2=0
\right\},
\qquad
g=-f'.
\end{align*}
Thus $\mathcal K_0$ is an infinite-dimensional kernel of the
linearized vorticity operator, and the semigroup exhibits no decay on
this subspace. In particular, it is not strongly stable on any phase
space containing $\mathcal K_0$.

The complementary subspace
\begin{align*}
\ker \mathbb P_0
=
\Big\{
\omega:
\widehat{\omega}_0(x_2)=0
\ \text{for a.e. }x_2
\Big\},
\qquad
\mathbb{P}_0\omega
=
\frac1L\int_{\mathbb T_L}\omega(x_1,\cdot)\,dx_1,
\end{align*}
is invariant under $e^{t\Lan }$, and all decay and stability
statements below are understood on this subspace. If the horizontal
mean velocity is not prescribed, the velocity kernel is enlarged by
the constant horizontal mode $\operatorname{span}\{(1,0)\}$, or,
equivalently, by allowing arbitrary periodic $f$ in the definition of
$\mathcal S_0$.
\end{theorem}

\vskip .1in

When $0<|k|<1$, the coercive energy argument fails because
$\lambda_0^{(k)}<0$. The scalar nature of the horizontal dissipation
nevertheless gives an exact spectral translation, reducing the viscous
problem to the inviscid spectral bound.
\begin{theorem}[Exact spectral shift and sharp instability criterion]\label{thm:B}
Fix $k\ne0$ and let $\mathcal A_{k,\nu}=-\nu k^2 \mathcal{I}+a\mathcal{M}_k$ on $\ell^2(\Z)$, with $\mathcal{M}_k$ as in \eqref{Mk}. Define the inviscid spectral bound by
\begin{align*}
\Lambda(k)
:=
s(\mathcal{M}_k)
:=
\sup\left\{
\operatorname{Re}z:
z\in\sigma(\mathcal{M}_k)
\right\}.
\end{align*}
 Then:
\begin{enumerate}
\item[$\mathrm{(i)}$] $\spec(\mathcal A_{k,\nu}) = a\,\spec(\mathcal{M}_k)-\nu k^2$, and the semigroup growth bound is exact:
\begin{align*}
\lim_{t\to\infty}\frac1t\log\big\|e^{t\mathcal A_{k,\nu}}\big\|_{\ell^2\to\ell^2}
= a\,\Lambda(k)-\nu k^2.
\end{align*}
\vskip .1in
\item[$\mathrm{(ii)}$]
For the normalized inviscid $k$-th mode, the spectral bound
$\Lambda(k)$ equals the asymptotic operator-norm growth exponent:
\begin{align*}
\lim_{\tau\to\infty}
\frac{1}{\tau}
\log
\bigl\|e^{\tau \mathcal{M}_k}\bigr\|_{\mathcal B(\ell^2)}
=
\Lambda(k).
\end{align*}
Moreover,
\begin{align*}
\Lambda(-k)=\Lambda(k),
\qquad
\Lambda(k)=0
\quad\text{for }|k|\geq1.
\end{align*}

\vskip .1in
\item[$\mathrm{(iii)}$]
The $k$-th mode is spectrally unstable if and only if
\begin{align*}
a\Lambda(k)>\nu k^2.
\end{align*}
Whenever this condition holds, the operator
$\mathcal A_{k,\nu}$ possesses an eigenvalue in the open right
half-plane and hence an exponentially growing normal mode. Its maximal
modal growth exponent, which also equals the asymptotic operator-norm
growth exponent, is
\begin{align*}
a\Lambda(k)-\nu k^2.
\end{align*}

\vskip .1in
\item[$\mathrm{(iv)}$] For $0<|k|<1$ one has $\Lambda(k)>0$, by the classical inviscid instability of the Kolmogorov flow {\cite{MeshalkinSinai,FriedlanderStraussVishik,FriedlanderHoward,BelenkayaFriedlanderYudovich,Lin}}. Consequently, if $L>2\pi$, then for each admissible wavenumber $0<|k|<1$ the flow is linearly unstable for every
\begin{align*}
a > \frac{\nu k^2}{\Lambda(k)} ,
\end{align*}
and the growth rate increases linearly in $a$.
\end{enumerate}
\end{theorem}

\vskip .1in
The preceding modewise estimates yield quantitative bounds for the
full linearized semigroup on the invariant complement of the
horizontally independent mode. Since the admissible horizontal
wavenumbers are determined by $L$, they lead to the following
domain-size classification.

\begin{corollary}[Dichotomy by domain size]\label{cor:L}
Let
$
k_{\min}:=\frac{2\pi}{L},
$
and let $\mathbb P_0$ denote the projection onto the horizontally
independent component.

\begin{enumerate}
\item[$\mathrm{(a)}$]
If $L<2\pi$, then
\begin{align*}
\left\|
e^{t\Lan }
\big|_{\ker\mathbb P_0}
\right\|_{\mathcal B(\ker\mathbb P_0)}
\leq
C_L e^{-\nu k_{\min}^2t},
\qquad t\geq0,
\end{align*}
where
\begin{align*}
C_L
=
\left(
1-\frac{1}{k_{\min}^2}
\right)^{-1/2}
=
\left(
1-\left(\frac{L}{2\pi}\right)^2
\right)^{-1/2}.
\end{align*}
Moreover,
\begin{align*}
\lim_{t\to\infty}
\frac1t
\log
\left\|
e^{t\Lan }
\big|_{\ker\mathbb P_0}
\right\|_{\mathcal B(\ker\mathbb P_0)}
=
-\nu k_{\min}^2
=
-\nu\left(\frac{2\pi}{L}\right)^2.
\end{align*}

\vskip .1in

\item[$\mathrm{(b)}$]
If $L=2\pi$, then there exists an absolute constant $C>0$ such
that
\begin{align*}
\left\|
e^{t\Lan }
\big|_{\ker\mathbb P_0}
\right\|_{\mathcal B(\ker\mathbb P_0)}
\leq
C(1+\sqrt{at})e^{-\nu t},
\qquad t\geq0.
\end{align*}
In particular,
\begin{align*}
\lim_{t\to\infty}
\frac1t
\log
\left\|
e^{t\Lan }
\big|_{\ker\mathbb P_0}
\right\|_{\mathcal B(\ker\mathbb P_0)}
=
-\nu.
\end{align*}
The algebraic factor is generated by the critical sectors
$|k|=1$; all sectors $|k|\geq2$ decay at strictly faster
exponential rates.

\vskip .1in

\item[$\mathrm{(c)}$]
If $L>2\pi$, define the finite nonempty set of admissible long-wave
wavenumbers by
\begin{align*}
\mathcal K_L
:=
\left\{
k\in\frac{2\pi}{L}\mathbb Z:
0<|k|<1
\right\}
\end{align*}
and set
\begin{align*}
a_{\mathrm c}(L,\nu)
:=
\min_{k\in\mathcal K_L}
\frac{\nu k^2}{\Lambda(k)}.
\end{align*}
Then the semigroup restricted to $\ker\mathbb P_0$ is exponentially
unstable in operator norm if and only if
\begin{align*}
a>a_{\mathrm c}(L,\nu).
\end{align*}
More precisely, whenever $a>a_{\mathrm c}(L,\nu)$,
\begin{align*}
\lim_{t\to\infty}
\frac1t
\log
\left\|
e^{t\Lan }
\big|_{\ker\mathbb P_0}
\right\|_{\mathcal B}
=
\max_{k\in\mathcal K_L}
\bigl(a\Lambda(k)-\nu k^2\bigr)
>0.
\end{align*}
In this case, the maximal growth exponent is realized by an
exponentially growing normal mode in one of the long-wave sectors.
\end{enumerate}
\end{corollary}

{The proof of Corollary~\ref{cor:L} is given in Section~\ref{sec:proofcor}.}

\medskip
\subsection{Discussion and comparison with full dissipation}

The preceding results show that mixing alone does not imply enhanced
dissipation; the decisive issue is whether diffusion acts on the
frequencies generated by the mixing. For the fully dissipative
Navier--Stokes equations, the viscous contribution to the $(k,n)$-th
Fourier mode is
\begin{align*}
-\nu(k^2+n^2)c_n.
\end{align*}
After extracting the common factor $e^{-\nu k^2t}$, the remaining
damping $-\nu n^2$ grows with the vertical frequency. Consequently,
shear-induced transfer toward large $|n|$ is met by increasingly
strong diffusion, producing the enhanced-dissipation rate
$e^{-c\sqrt{a\nu}\,t}$ for the Kolmogorov flow
{\cite{BeckWayne,LinXu,IbrahimMaekawaMasmoudi,WeiZhangZhao}; see also \cite{ConstantinKiselevRyzhikZlatos,BedrossianCotiZelati,CotiZelatiElgindiWidmayerCMP} for enhanced dissipation near other shear flows}. Together with
inviscid damping and vorticity depletion, this mechanism also plays a
central role in nonlinear stability
\cite{LiWeiZhang,ChenJiaWeiZhang}.

The situation under horizontal dissipation is fundamentally different.
On each fixed horizontal mode $k$, every vertical mode is damped at
the same rate $\nu k^2$, independently of $n$. Thus the vertical
frequency cascade generated by the shear does not strengthen the
dissipation. This obstruction is expressed by the exact factorization
\begin{align*}
e^{t\mathcal A_{k,\nu}}
=
e^{-\nu k^2t}e^{at\mathcal{M}_k}.
\end{align*}
In the stable range $|k|>1$, the inviscid group $e^{at\mathcal{M}_k}$ is
uniformly bounded after symmetrization, and Theorem~\ref{thm:A} yields
sharp two-sided bounds at the bare diffusive rate. Hence enhanced
dissipation is absent in the strongest possible sense: no uniform
improvement over $e^{-\nu k^2t}$ is possible, regardless of the shear
amplitude $a$.

For reference, the qualitative modewise classification is summarized
in the following table:

\begin{align*}
\begin{array}{|c|c|}
\hline
\text{horizontal wavenumber}
&
\text{linear behavior}
\\
\hline
|k|>1
&
\text{sharp diffusive decay and no enhanced dissipation}
\\
\hline
|k|=1
&
\text{algebraic transient growth followed by diffusive decay}
\\
\hline
k=0
&
\text{infinite-dimensional undamped shear kernel}
\\
\hline
0<|k|<1
&
\text{long-wave instability above an exact threshold}
\\
\hline
\end{array}
\end{align*}

This classification also indicates that the natural nonlinear problem
is stability modulo the manifold of shear equilibria. In particular,
one should not expect convergence to the original Kolmogorov profile
itself, but rather convergence, if it holds, to a nearby shear flow.
The nonlinear stability problem and the comparison with dissipation
acting transversely to the shear are left for future investigation.

\subsection{The gap between the linear and nonlinear problems}

The exact factorization
$$
e^{t\mathcal A_{k,\nu}}
=
e^{-\nu k^2t}e^{at\mathcal{M}_k}
$$
is a sectorwise feature of the linearized equation and does not extend
directly to the nonlinear perturbation problem. Indeed, the nonlinear
vorticity term $u\cdot\nabla\omega$ couples distinct horizontal
Fourier sectors and may transfer energy simultaneously in the
horizontal and vertical frequencies. In particular, interactions
between the modes $k$ and $-k$ can generate a horizontally
independent component, which lies in the undamped kernel described in
Theorem~\ref{thm:k0}.

A second difficulty is the absence of vertical regularization. At the
vorticity $L^2$ level, incompressibility gives the formal cancellation
$$
\int_{\mathbb T_L\times\mathbb T_{2\pi}}
(u\cdot\nabla\omega)\,\omega\,dx
=
0,
$$
but the corresponding high-order estimates contain commutators
involving vertical derivatives of both $u$ and $\omega$. The
dissipation supplies only
$$
\nu\|\partial_1\omega\|_{L^2}^2
$$
and provides no direct control of $\partial_2\omega$. This is in
sharp contrast with the fully dissipative problem, where elliptic
smoothing and enhanced dissipation provide both derivative control and
time-integrability of the non-shear component.

Consequently, the modewise $L^2$ bounds established here do not by
themselves imply nonlinear stability, nor do their constants yield a
uniform in $a$ nonlinear threshold in higher order norms. A nonlinear
theory would have to combine anisotropic commutator estimates with a
modulation argument along the manifold of shear equilibria, while
controlling the quadratic transfer into the undamped mode. The
present linear analysis identifies the decay, transient-growth, and
spectral instability mechanisms { that any such nonlinear theory must accommodate.}

\subsection{Organization of the paper}

In Section~2, we derive the horizontal and vertical Fourier reductions
and identify the operator structure behind
\eqref{factorization-intro}. In Section~3, we prove the sharp
diffusive bounds for $|k|>1$, the critical $\sqrt{t}$ growth law for
$|k|=1$, the description of the horizontally independent kernel, and
the exact spectral-shift formula in the unstable range{, and we deduce the domain-size classification of Corollary~\ref{cor:L}}. In Section~4,
we discuss the nonlinear questions suggested by the linear theory.

\subsection{Notation}
Throughout the paper,
$
k\in\frac{2\pi}{L}\mathbb Z
$
denotes the horizontal Fourier wavenumber and $n\in\mathbb Z$ the
vertical Fourier index. The coefficients $\lambda_n^{(k)}$ are defined
in \eqref{lambdadef}; the superscript $(k)$ is omitted when there is
no risk of confusion. We write $\|\cdot\|$ for either the
$L^2(\mathbb T_{2\pi})$ norm or, by Parseval's identity, the
corresponding $\ell^2(\mathbb Z)$ norm.

\medskip
\section{Fourier reduction and the structure of the advection operator}
\label{sec:reduction}

We first justify the Fourier reduction used above. Applying
$\nabla^\perp\cdot$ to \eqref{linvel} eliminates the pressure and yields
the vorticity equation \eqref{linvort}. Conversely, on each mode
$k\neq0$, the velocity is recovered from the vorticity by
\begin{align*}
\widehat u_k
=
-\nabla_k^\perp
\bigl(k^2-\partial_2^2\bigr)^{-1}\widehat\omega_k,
\end{align*}
where
\begin{align*}
\nabla_k^\perp:=(-\partial_2,ik).
\end{align*}
Since the corresponding Fourier multiplier is bounded for fixed $k\neq0$,
estimates for the vorticity transfer to the velocity, with constants depending
on $k$. The mode $k=0$ is treated separately in
Theorem~\ref{thm:k0}.

Because the coefficients in \eqref{linvort} are independent of $x_1$,
the horizontal Fourier sectors are invariant. Thus, for
\begin{align*}
\omega(x_1,x_2,t)
=
\sum_{k\in(2\pi/L)\mathbb Z}
e^{ikx_1}\widehat\omega_k(x_2,t),
\end{align*}
the $k$-th mode satisfies \eqref{modek}.

Fix $k\neq0$. Using
\begin{align*}
\sin x_2
=
\frac{e^{ix_2}-e^{-ix_2}}{2i},
\qquad
\bigl(k^2-\partial_2^2\bigr)^{-1}e^{inx_2}
=
\frac{1}{k^2+n^2}e^{inx_2},
\end{align*}
and
\begin{align*}
\widehat\omega_k(x_2,t)
=
\sum_{n\in\mathbb Z}c_n(t)e^{inx_2},
\end{align*}
we obtain
\begin{align*}
\left[
1-\bigl(k^2-\partial_2^2\bigr)^{-1}
\right]\widehat\omega_k
=
\sum_{n\in\mathbb Z}
\lambda_n c_n(t)e^{inx_2},
\qquad
\lambda_n:=1-\frac{1}{k^2+n^2}.
\end{align*}
Moreover,
\begin{align*}
\begin{aligned}
\sin x_2
\left[
1-\bigl(k^2-\partial_2^2\bigr)^{-1}
\right]\widehat\omega_k
&=
\frac{1}{2i}
\sum_{n\in\mathbb Z}
\left(
\lambda_{n-1}c_{n-1}(t)
-
\lambda_{n+1}c_{n+1}(t)
\right)e^{inx_2}.
\end{aligned}
\end{align*}
Comparing the coefficients of $e^{inx_2}$ in \eqref{modek} gives
\begin{align*}
\frac{d}{dt}c_n(t)
+
\frac{ka}{2}
\left(
\lambda_{n-1}c_{n-1}(t)
-
\lambda_{n+1}c_{n+1}(t)
\right)
=
-\nu k^2c_n(t),
\end{align*}
This is precisely \eqref{tridiag}.

To write the system in operator form, define the right-shift operator
$\mathcal{S}$ and the diagonal operator $\mathcal{D}$ on $\ell^2(\mathbb Z)$ by
\begin{align*}
(\mathcal{S}c)_n:=c_{n-1},
\qquad
(\mathcal{D}c)_n:=\lambda_n c_n.
\end{align*}
Then
\begin{align*}
(\mathcal{S}\mathcal{D}\mathcal{}c)_n=\lambda_{n-1}c_{n-1},
\qquad
(\mathcal{S}^{-1}\mathcal{D}c)_n=\lambda_{n+1}c_{n+1},
\end{align*}
where we used $\mathcal{S}\mathcal{D}=\,$``multiply by $\lambda$, then shift up'', matching \eqref{Mk}.

Hence
\begin{align}\label{operatorform}
\mathcal A_{k,\nu}
=
-\nu k^2\mathcal{I}+a\mathcal{M}_k,
\qquad
\mathcal{M}_k
=
-\frac{k}{2}\bigl(\mathcal{S}\mathcal{D}-\mathcal{S}^{-1}\mathcal{D}\bigr).
\end{align}
Since $\mathcal{S}$ and $\mathcal{S}^{-1}$ are unitary on $\ell^2(\mathbb Z)$, while
\begin{align*}
\|\mathcal{D}\|_{\ell^2\to\ell^2}
=
\sup_{n\in\mathbb Z}|\lambda_n|<\infty,
\end{align*}
the operator $\mathcal{M}_k$ is bounded on $\ell^2(\mathbb Z)$. Consequently,
\begin{align*}
e^{t\mathcal A_{k,\nu}}
=
e^{-\nu k^2t}e^{at\mathcal{M}_k},
\qquad t\geq0,
\end{align*}
where both exponentials are defined by norm-convergent power series. This
factorization is the operator form of the substitution
\eqref{substitution}. It expresses the central structural feature of the
model: on each fixed horizontal Fourier mode, the dissipative operator is the
scalar $-\nu k^2\mathcal{I}$, and hence commutes with the inviscid advection operator
$a\mathcal{M}_k$.

The following elementary lemma isolates the algebraic structure used in the stable range.

\begin{lemma}[Skew-symmetrization]\label{lem:skew}
Suppose $\lambda_n>0$ for all $n\in\Z$, and let $\mathcal{W}:=\diag(\lambda_n^{-1/2})$. Then
\begin{align*}
\mathcal{W}^{-1}\mathcal{M}_k\mathcal{W} = -\frac{k}{2}\,\big(\Sigma-\Sigma^{*}\big),
\qquad
(\Sigma c)_n := \big(\lambda_{n-1}\lambda_{n}\big)^{1/2}c_{n-1},
\end{align*}
is a bounded, skew-adjoint operator on $\ell^2(\Z)$. If in addition $\inf_n\lambda_n>0$, then $\mathcal{W}$ is bounded with bounded inverse, and consequently $\spec(\mathcal{M}_k)\subset i\R$ and $\sup_{t\in\R}\|e^{t\mathcal{M}_k}\|<\infty$.
\end{lemma}
\begin{proof}
Let $c_{00}(\mathbb Z)$ denote the finitely supported sequences.
On $c_{00}(\mathbb Z)$,
\begin{align*}
(\mathcal{W}^{-1}\mathcal{M}_k\mathcal{W})_{n,m}
=
\lambda_n^{1/2}(\mathcal{M}_k)_{n,m}\lambda_m^{-1/2}.
\end{align*}
Since
\begin{align*}
(\mathcal{M}_k)_{n,n-1}=-\frac{k}{2}\lambda_{n-1},
\qquad
(\mathcal{M}_k)_{n,n+1}=\frac{k}{2}\lambda_{n+1},
\end{align*}
it follows that
\begin{align*}
(\mathcal{W}^{-1}\mathcal{M}_k\mathcal{W}c)_n
=
-\frac{k}{2}
\left[
(\lambda_{n-1}\lambda_n)^{1/2}c_{n-1}
-
(\lambda_n\lambda_{n+1})^{1/2}c_{n+1}
\right].
\end{align*}
Defining
$
(\Sigma c)_n
=
(\lambda_{n-1}\lambda_n)^{1/2}c_{n-1},
$
we have
$
(\Sigma^*c)_n
=
(\lambda_n\lambda_{n+1})^{1/2}c_{n+1},
$
and hence
\begin{align*}
\mathcal{W}^{-1}\mathcal{M}_k\mathcal{W}
=
-\frac{k}{2}(\Sigma-\Sigma^*)
\qquad\text{on }c_{00}(\mathbb Z).
\end{align*}

Moreover,
\begin{align*}
\|\Sigma c\|_{\ell^2}^2
=
\sum_{n\in\mathbb Z}
\lambda_{n-1}\lambda_n|c_{n-1}|^2
\le
\left(\sup_n\lambda_n\right)^2
\|c\|_{\ell^2}^2.
\end{align*}
Thus
\begin{align*}
\|\Sigma\|_{\ell^2\to\ell^2}
\le
\sup_n\lambda_n
\le 1,
\end{align*}
so the preceding identity extends uniquely to a bounded operator
\begin{align*}
\mathcal{Q}:=-\frac{k}{2}(\Sigma-\Sigma^*)
\quad\text{on }\ell^2(\mathbb Z).
\end{align*}
Since
\begin{align*}
\mathcal{Q}^*
=
-\frac{k}{2}(\Sigma^*-\Sigma)
=
-\mathcal{Q},
\end{align*}
the operator $\mathcal{Q}$ is skew-adjoint.

If
$
\lambda_*:=\inf_{n\in\mathbb Z}\lambda_n>0,
$
then $\mathcal{W}$ and $\mathcal{W}^{-1}$ are bounded, with
\begin{align*}
\|\mathcal{W}\|=\lambda_*^{-1/2},
\qquad
\|\mathcal{W}^{-1}\|
=
\left(\sup_n\lambda_n\right)^{1/2}
\le 1.
\end{align*}
Therefore
\begin{align*}
\mathcal{Q}=\mathcal{W}^{-1}\mathcal{M}_k\mathcal{W},
\qquad
\mathcal{M}_k=\mathcal{W}\mathcal{Q}\mathcal{W}^{-1},
\end{align*}
and bounded similarity yields
\begin{align*}
\sigma(\mathcal{M}_k)=\sigma(\mathcal{Q})\subset i\mathbb R.
\end{align*}
Finally,
\begin{align*}
e^{t\mathcal{M}_k}=\mathcal{W}e^{t\mathcal{Q}}\mathcal{W}^{-1},
\qquad
\|e^{t\mathcal{Q}}\|_{\ell^2\to\ell^2}=1,
\quad t\in\mathbb R,
\end{align*}
whence
\begin{align*}
\sup_{t\in\mathbb R}\|e^{t\mathcal{M}_k}\|_{\ell^2\to\ell^2}
\le
\|\mathcal{W}\|\,\|\mathcal{W}^{-1}\|
\le
\lambda_*^{-1/2}.
\end{align*}
\end{proof}

\begin{remark}
The weight $\mathcal{W}$ has a transparent meaning: $\sum_n\lambda_n|c_n|^2$ is (up to the sign of $\lambda_0$ in the unstable range) the quadratic form of $1+\Delta^{-1}$ on the mode $k$, i.e.\ the classical Arnold-type energy--Casimir functional $\int \omega(\omega+\Delta^{-1}\omega)$ {(see, e.g., \cite{LiYanguang})} restricted to the $k$-th mode. The conservation law below is thus the linearized energy--Casimir identity, which the anisotropic dissipation respects exactly rather than merely dissipating.
\end{remark}

The preceding operator representation also identifies the essential
spectrum of $\mathcal{M}_k$. This observation will be useful in the long-wave
regime, since it shows that every spectral point off the imaginary
axis is a discrete eigenvalue rather than a point of the continuous
spectrum.

\begin{lemma}[Essential spectrum]\label{lem:essential-spectrum}
For every $k\neq0$, the essential spectrum of $\mathcal{M}_k$ is
\begin{align*}
\sigma_{\mathrm{ess}}(\mathcal{M}_k)
=
i[-|k|,|k|].
\end{align*}
Here $\sigma_{\mathrm{ess}}$ denotes the Fredholm essential spectrum.
Consequently, every spectral point of $\mathcal{M}_k$ outside
$i[-|k|,|k|]$, and in particular every spectral point with nonzero
real part, is an isolated eigenvalue of finite algebraic multiplicity.
\end{lemma}

\begin{proof}
Writing $\mathcal{D}=\mathcal{I}+(\mathcal{D}-\mathcal{I})$ in the operator representation
\eqref{operatorform}, define the limiting Laurent operator
\begin{align*}
\mathcal{M}_k^\infty
:=
-\frac{k}{2}\bigl(\mathcal{S}-\mathcal{S}^{-1}\bigr).
\end{align*}
Then
\begin{align*}
\mathcal{M}_k-\mathcal{M}_k^\infty
=
-\frac{k}{2}
\left[
\mathcal{S}(\mathcal{D}-\mathcal{I})-\mathcal{S}^{-1}(\mathcal{D}-\mathcal{I})
\right].
\end{align*}
Since
\begin{align*}
\mathcal{D}-\mathcal{I}
=
\operatorname{diag}
\left(
-\frac{1}{k^2+n^2}
\right)_{n\in\mathbb Z}
\end{align*}
has diagonal entries tending to zero as $|n|\to\infty$, it is
compact on $\ell^2(\mathbb Z)$. Hence $\mathcal{M}_k-\mathcal{M}_k^\infty$ is compact.

Under the discrete Fourier transform on $\ell^2(\mathbb Z)$, the
operator $\mathcal{M}_k^\infty$ is unitarily equivalent to multiplication on
$L^2(\mathbb T)$ by
\begin{align*}
-\frac{k}{2}
\left(e^{i\theta}-e^{-i\theta}\right)
=
-ik\sin\theta.
\end{align*}
Therefore,
\begin{align*}
\sigma(\mathcal{M}_k^\infty)
=
\sigma_{\mathrm{ess}}(\mathcal{M}_k^\infty)
=
\operatorname{ess\,ran}(-ik\sin\theta)
=
i[-|k|,|k|].
\end{align*}
By invariance of the Fredholm essential spectrum under compact
perturbations,
\begin{align*}
\sigma_{\mathrm{ess}}(\mathcal{M}_k)
=
\sigma_{\mathrm{ess}}(\mathcal{M}_k^\infty)
=
i[-|k|,|k|].
\end{align*}

It remains to identify the spectral points outside this interval.
For $z\notin i[-|k|,|k|]$, the operator
$\mathcal{M}_k^\infty-z\mathcal{I}$ is invertible, and
\begin{align*}
\mathcal{M}_k-z\mathcal{I}
=
\left[
\mathcal{I}+(\mathcal{M}_k-\mathcal{M}_k^\infty)(\mathcal{M}_k^\infty-z\mathcal{I})^{-1}
\right]
(\mathcal{M}_k^\infty-z\mathcal{I}).
\end{align*}
The first factor is an analytic Fredholm family of index zero on
$\mathbb C\setminus i[-|k|,|k|]$, being the identity plus a compact
operator. It is invertible for $|z|$ sufficiently large; hence the
analytic Fredholm theorem implies that its noninvertibility set is
discrete and that the corresponding characteristic values have finite
algebraic multiplicity. Since the Fredholm index is zero, every such
noninvertibility point has a nontrivial kernel. Consequently, every
spectral point of $\mathcal{M}_k$ outside $i[-|k|,|k|]$ is an isolated
eigenvalue of finite algebraic multiplicity.
\end{proof}

\begin{corollary}[Realization of the spectral bound]\label{cor:unstable-eigenmode}
Fix $k\neq0$ and suppose that $\Lambda(k)>0$. Then there exist
$z_k\in\sigma_{\mathrm p}(\mathcal{M}_k)$ and
$c_k\in\ell^2(\mathbb Z)\setminus\{0\}$ such that
\begin{align*}
\mathcal{M}_kc_k=z_kc_k,
\qquad
\operatorname{Re}z_k=\Lambda(k).
\end{align*}
Moreover, $z_k$ is an isolated eigenvalue of finite algebraic
multiplicity. Since
\begin{align*}
\mathcal A_{k,\nu}\,c_k
=
\bigl(az_k-\nu k^2\bigr)c_k,
\end{align*}
the $k$-th viscous mode possesses an exponentially growing normal
mode if and only if
\begin{align*}
a\Lambda(k)>\nu k^2.
\end{align*}
In that case, its maximal modal growth exponent is
$
a\Lambda(k)-\nu k^2.
$
\end{corollary}

\begin{proof}
The compactness of $\sigma(\mathcal{M}_k)$ implies that its spectral bound is
attained at some $z_k\in\sigma(\mathcal{M}_k)$. Since
\begin{align*}
\operatorname{Re}z_k=\Lambda(k)>0,
\end{align*}
we have $z_k\notin\sigma_{\mathrm{ess}}(\mathcal{M}_k)=i[-|k|,|k|]$.
Lemma~\ref{lem:essential-spectrum} therefore shows that $z_k$ is an
isolated eigenvalue of finite algebraic multiplicity. The remaining
claims follow from
$
\mathcal A_{k,\nu}=a\mathcal{M}_k-\nu k^2\mathcal{I}.
$
\end{proof}

\vskip .1in
\section{Proofs of Theorems \ref{thm:A}--\ref{thm:B}}
\subsection{The stable range $|k|>1$: proof of Theorem \ref{thm:A}}
\begin{proof}[Proof of Theorem~\ref{thm:A}]
Fix $|k|>1$, and write $\lambda_n=\lambda_n^{(k)}$. Then
\begin{align*}
\lambda_*:=1-\frac{1}{k^2}
=\lambda_0
\leq \lambda_n<1,
\qquad n\in\mathbb Z,
\end{align*}
so that $\lambda_*>0$.

Define $d(t):=e^{\nu k^2t}c(t),$
then
\begin{align*}
\frac{d}{dt}d(t)=a\mathcal{M}_kd(t).
\end{align*}
Set
\begin{align*}
\mathcal{W}=\operatorname{diag}(\lambda_n^{-1/2}),
\qquad
h(t):=\mathcal{W}^{-1}d(t),
\qquad
h_n(t)=\lambda_n^{1/2}d_n(t).
\end{align*}
Since $\mathcal{W}$ is independent of $t$, a direct computation  yields
\begin{align*}
\frac{d}{dt}h(t)
={a\mathcal{W}^{-1}\mathcal{M}_k\mathcal{W}} h(t)
=
a\mathcal{G}h(t),
\qquad
\mathcal{G}:=\mathcal{W}^{-1}\mathcal{M}_k\mathcal{W}.
\end{align*}
By Lemma~\ref{lem:skew},
$
\mathcal{G}^*=-\mathcal{G}.
$
Hence,
\begin{align*}
\begin{aligned}
\frac{d}{dt}\|h(t)\|_{\ell^2}^2
&=
2a\,\operatorname{Re}
\left\langle \mathcal{G}h(t),h(t)\right\rangle_{\ell^2}=0.
\end{aligned}
\end{align*}
Therefore, we obtain the conservation law
\begin{align*}
\sum_{n\in\mathbb Z}\lambda_n|d_n(t)|^2
=
\sum_{n\in\mathbb Z}\lambda_n|d_n(0)|^2,
\qquad t\geq0.
\end{align*}
Using $d(t)=e^{\nu k^2t}c(t)$ once again, this is equivalent to
\begin{align*}
e^{2\nu k^2t}
\sum_{n\in\mathbb Z}\lambda_n|c_n(t)|^2
=
\sum_{n\in\mathbb Z}\lambda_n|c_n(0)|^2.
\end{align*}

Consequently,
\begin{align*}
\lambda_*\|d(t)\|_{\ell^2}^2
\leq
\sum_n\lambda_n|d_n(t)|^2
=
\sum_n\lambda_n|d_n(0)|^2
\leq
\|d(0)\|_{\ell^2}^2,
\end{align*}
which gives
\begin{align*}
\|d(t)\|_{\ell^2}
\leq
\lambda_*^{-1/2}\|d(0)\|_{\ell^2}.
\end{align*}
Conversely,
\begin{align*}
\|d(t)\|_{\ell^2}^2
\geq
\sum_n\lambda_n|d_n(t)|^2
=
\sum_n\lambda_n|d_n(0)|^2
\geq
\lambda_*\|d(0)\|_{\ell^2}^2,
\end{align*}
and hence
\begin{align*}
\|d(t)\|_{\ell^2}
\geq
\lambda_*^{1/2}\|d(0)\|_{\ell^2}.
\end{align*}

Returning to $c(t)=e^{-\nu k^2t}d(t)$, we obtain
\begin{align*}
\lambda_*^{1/2}e^{-\nu k^2t}\|c(0)\|_{\ell^2}
\le
\|c(t)\|_{\ell^2}
\le
\lambda_*^{-1/2}e^{-\nu k^2t}\|c(0)\|_{\ell^2}.
\end{align*}
Substituting $\lambda_*=1-k^{-2}$ yields \eqref{twosided}. The constants
are independent of $a$, $\nu$, and $t$. In particular, the lower bound
rules out any uniform decay estimate on the $k$-th mode with an exponential
rate strictly larger than $\nu k^2$.
\end{proof}

\begin{remark}[Sharpness of the diffusive rate]
The lower bound in \eqref{twosided} implies that, for every
$|k|>1$,
\begin{align*}
\bigl\|e^{t\mathcal A_{k,\nu}}\bigr\|_{\mathcal B(\ell^2)}
\geq
\left(1-\frac{1}{k^2}\right)^{1/2}
e^{-\nu k^2t},
\qquad t\geq0.
\end{align*}
Consequently, no estimate of the form
\begin{align*}
\bigl\|e^{t\mathcal A_{k,\nu}}\bigr\|_{\mathcal B(\ell^2)}
\leq
C e^{-\gamma t},
\qquad
\gamma>\nu k^2,
\end{align*}
can hold uniformly in time, irrespective of the shear amplitude $a$.
Together with the upper bound in \eqref{twosided}, this yields
\begin{align*}
\lim_{t\to\infty}
\frac1t
\log
\bigl\|e^{t\mathcal A_{k,\nu}}\bigr\|_{\mathcal B(\ell^2)}
=
-\nu k^2.
\end{align*}
Thus the bare diffusive rate is sharp in the fixed-mode vorticity
norm.
\end{remark}

\begin{remark}[Uniformity in $a$ and inviscid limit]
Since the constants in \eqref{twosided} do not depend on $\nu$, Theorem~\ref{thm:A} contains the inviscid statement: for $\nu=0$ and $|k|>1$ the linearized Euler dynamics on the $k$-th mode is bounded above and below, uniformly in time --- the linearized energy--Casimir stability of Kolmogorov flow at high horizontal wavenumbers --- and the viscous dynamics is exactly this bounded group multiplied by $e^{-\nu k^2t}$.
\end{remark}

\subsection{The marginal case $|k|=1$: proof of Theorem
\ref{thm:k1}}
The proof is divided into three steps. We first reduce the critical
dynamics to two half-line systems coupled to the zero mode. We then
compute the zero-energy spectral density of the associated Jacobi
operator, and finally use this information to derive the sharp
semigroup bounds.

\begin{proof}
For $ |k|=1 $, the weight degenerates at exactly one site:
\begin{align*}
\lambda_0=0,
\qquad
\lambda_{\pm1}=\frac12,
\qquad
\lambda_n=1-\frac{1}{1+n^2}\geq\frac12
\quad (n\neq0).
\end{align*}

We first consider $k=1$. Since
$
\mathcal{M}_{-1}=-\mathcal{M}_1,
$
while the unitary reflection
$
(\mathcal{R}c)_n=c_{-n}
$
satisfies
$
\mathcal{R}\mathcal{M}_1\mathcal{R}^{-1}=-\mathcal{M}_1,
$
we have
\begin{align*}
\|e^{\tau \mathcal{M}_{-1}}\|
=
\|e^{-\tau \mathcal{M}_1}\|
=
\|Re^{\tau \mathcal{M}_1}\mathcal{R}^{-1}\|
=
\|e^{\tau \mathcal{M}_1}\|.
\end{align*}
It is therefore enough to analyze $\mathcal{M}_1$.

For the inviscid analysis, introduce the rescaled time
$
\tau=at
$
and set
$
d(\tau)
=
e^{(\nu/a)\tau}c(\tau/a).
$
Then
\begin{align*}
\partial_\tau d=\mathcal{M}_1d.
\end{align*}

\medskip
\noindent
\textbf{Step 1: Half-line reduction.}
We first separate the positive and negative vertical modes and isolate
their coupling to the neutral coefficient $d_0$.

Since $\lambda_0=0$, the positive and negative nonzero modes evolve
autonomously. For $n\geq1$, define
\begin{align*}
h_n^+=\sqrt{\lambda_n}\,d_n,
\qquad
h_n^-=\sqrt{\lambda_n}\,d_{-n}.
\end{align*}
A direct computation gives
\begin{align}
\label{half-line-evolution}
\partial_\tau h^+=\mathcal{G}h^+,
\qquad
\partial_\tau h^-=-\mathcal{G}h^-,
\end{align}
where $\mathcal{G}$ is the bounded skew-adjoint operator on
$\ell^2(\mathbb N)$ given by
\begin{align*}
(\mathcal{G}h)_n
=
-b_{n-1}h_{n-1}+b_nh_{n+1},
\qquad b_0=0,
\end{align*}
with
\begin{align}
\label{critical-jacobi-coefficients}
b_n
=
\frac12\sqrt{\lambda_n\lambda_{n+1}}
=
\frac12
\frac{n(n+1)}
{\sqrt{(n^2+1)((n+1)^2+1)}}.
\end{align}
The zero mode satisfies
\begin{align}
\label{zero-mode-forcing}
\partial_\tau d_0
=
\frac14(d_1-d_{-1})
=
\beta(h_1^+-h_1^-),
\qquad
\beta=\frac{1}{2\sqrt2}.
\end{align}

Let
$
\mathcal{D}_+=\operatorname{diag}(\lambda_n)_{n\geq1}.
$
The map
$
d\longmapsto(d_0,h^+,h^-)
$
is an isometry from $\ell^2(\mathbb Z)$ onto
\begin{align*}
\mathcal Y
=
\mathbb C\oplus\ell^2(\mathbb N)\oplus\ell^2(\mathbb N)
\end{align*}
when $\mathcal Y$ is equipped with the norm
\begin{align}
\label{X-norm}
\|(z,h^+,h^-)\|_{\mathcal Y}^2
=
|z|^2
+
\|\mathcal{D}_+^{-1/2}h^+\|_{\ell^2}^2
+
\|\mathcal{D}_+^{-1/2}h^-\|_{\ell^2}^2.
\end{align}
Indeed,
\begin{align*}
\frac12\leq\lambda_n<1,
\qquad n\geq1,
\end{align*}
so $\mathcal{D}_+^{1/2}$ and $\mathcal{D}_+^{-1/2}$ are both bounded. In particular,
the norm in \eqref{X-norm} is equivalent, with absolute constants,
to the standard product norm.

\medskip
\noindent

\medskip
\noindent
\textbf{Step 2: Zero-energy spectral density.}
We first explain the spectral mechanism behind the critical growth.
The degeneracy $\lambda_0=0$ separates the nonzero vertical modes
into two half-line systems, while the neutral coefficient $d_0$ is
forced by their boundary values $h_1^\pm$. After the unitary
conjugation below, each half-line generator becomes $i\mathcal J$,
where $\mathcal J$ is a self-adjoint Jacobi operator. Thus the
long-time accumulation of the forcing of $d_0$ is governed by the
spectral measure of $\mathcal J$ associated with the boundary vector
$e_1$.

More precisely, the relevant forcing terms have the form
\begin{align*}
g_\tau^\pm
=
\int_0^\tau e^{\mp s\mathcal G}e_1\,ds.
\end{align*}
If $\mu$ denotes the spectral measure of $\mathcal J$ associated
with $e_1$, then the spectral theorem gives
\begin{align*}
\|g_\tau^\pm\|_{\ell^2}^2
=
\int_{\mathbb R}
\left|
\frac{e^{i\tau x}-1}{x}
\right|^2
d\mu(x).
\end{align*}
Hence the behavior of $\mu$ near $x=0$ determines the growth of
the integrated boundary forcing. An atom at zero would produce growth
of order $\tau$ in $\|g_\tau^\pm\|$, whereas a continuous,
nonvanishing density at zero produces
\begin{align*}
\|g_\tau^\pm\|_{\ell^2}^2
\sim
2\pi w(0)\tau,
\end{align*}
and therefore growth of order $\sqrt{\tau}$. The purpose of the
Weyl--Jost analysis below is precisely to exclude a zero eigenvalue,
establish the regularity of the absolutely continuous spectral density
at zero, and compute its value:
\begin{align*}
w(0)=\frac4\pi.
\end{align*}
This calculation yields both the exponent $1/2$ and the sharp
constant in the asymptotic law for $e^{\tau \mathcal{M}_k}$.

We now carry out this analysis. Define the unitary operator
\begin{align*}
(\mathcal Uf)_n=i^nf_n.
\end{align*}
Then
\begin{align*}
\mathcal U^{-1}\mathcal G\mathcal U=i\mathcal J,
\end{align*}
where $\mathcal J$ is the self-adjoint half-line Jacobi operator
\begin{align}
\label{critical-Jacobi}
(\mathcal Jf)_n
=
b_{n-1}f_{n-1}+b_nf_{n+1}.
\end{align}
By \eqref{critical-jacobi-coefficients},
\begin{align*}
b_n=\frac12+O(n^{-2}),
\end{align*}
and hence
\begin{align*}
\sum_{n=1}^{\infty}
\left|b_n-\frac12\right|<\infty.
\end{align*}

{We use the Jost-solution theory for Jacobi operators whose coefficients converge to those of the free operator in $\ell^1$. We emphasize that the coefficients \eqref{critical-jacobi-coefficients} satisfy
\begin{align*}
b_n-\frac12=-\frac{1}{2n^2}+O(n^{-3}),
\end{align*}
so that $\sum_n|b_n-\frac12|<\infty$ but $\sum_n n|b_n-\frac12|=\infty$. Hence the first-moment hypothesis, which is used in the scattering theory up to the band edges $E=\pm1$ (see \cite[Chapter~10]{TeschlJacobi}), is not available. This is harmless here, because $E=0$ lies in the interior of the band $(-1,1)$, and at interior energies the $\ell^1$ condition suffices: the Jost solutions exist and depend continuously on $E+i0$ for $E\in(-1,1)$, the Weyl functions of $\mathcal{J}$ admit continuous boundary values there, and the spectral measure of $\mathcal{J}$ associated with $e_1$ is absolutely continuous near $E$, provided the corresponding Jost denominator does not vanish. It also identifies the boundary value of the Weyl function with the outgoing Jost solution. See \cite{DamanikSimon} and \cite{TeschlJacobi}.}

In this case the zero-energy Jost solution can be computed explicitly.
Define
\begin{align*}
f_n(0)
=
\frac{(-i)^n}{\sqrt{\lambda_n}},
\qquad n\geq1.
\end{align*}
Then, for every $n\geq2$,
\begin{align*}
b_{n-1}f_{n-1}(0)+b_nf_{n+1}(0)=0.
\end{align*}
Moreover,
\begin{align*}
f_n(0)=(-i)^n(1+o(1)),
\qquad n\to\infty.
\end{align*}
Thus $f(0)$ is the normalized outgoing zero-energy Jost solution for
the tail recurrence. No boundary condition at $n=1$ is imposed at this stage; the
half-line boundary condition is encoded in the Weyl denominator below.

For $n\geq1$, let $\mathcal J^{(n)}$ denote the restriction of
$\mathcal J$ to the tail
$
\{n,n+1,\ldots\},
$
with the Dirichlet boundary condition at $n-1$, and define its Weyl
function by
\begin{align*}
m_n(z)
=
\left\langle
e_n,(\mathcal J^{(n)}-z)^{-1}e_n
\right\rangle,
\qquad
\operatorname{Im}z>0.
\end{align*}
With this resolvent convention, the Schur complement identity gives
\begin{align}
\label{Weyl-Schur-recursion}
m_n(z)
=
-\frac{1}{z+b_n^2m_{n+1}(z)}.
\end{align}

Let $f(z)$ denote the outgoing Jost solution normalized by the
standard free asymptotics at infinity. With this normalization, its
boundary value at zero is
\begin{align*}
f_n(0+i0)=\frac{(-i)^n}{\sqrt{\lambda_n}}.
\end{align*}
 The standard Weyl-solution representation gives
\begin{align}
\label{Weyl-solution-representation}
m_{n+1}(z)
=
-\frac{f_{n+1}(z)}{b_nf_n(z)}.
\end{align}
Substituting \eqref{Weyl-solution-representation} into
\eqref{Weyl-Schur-recursion}, we obtain
\begin{align}
\label{m-function-Jost-representation}
m_n(z)
=
\frac{f_n(z)}
{b_nf_{n+1}(z)-zf_n(z)},
\qquad
\operatorname{Im}z>0.
\end{align}

The short-range Jacobi theory cited above ensures that the outgoing
Weyl solution and the Weyl function admit continuous boundary values
in a neighborhood of $E=0$, provided the denominator in
\eqref{m-function-Jost-representation} does not vanish. Therefore,
for such $E$,
\begin{align}
\label{m-boundary-Jost-representation}
m_n(E+i0)
=
\frac{f_n(E+i0)}
{b_nf_{n+1}(E+i0)-Ef_n(E+i0)}.
\end{align}
For brevity, we write
\begin{align*}
f_n(E):=f_n(E+i0)
\end{align*}
in what follows.

At $E=0$, the explicit outgoing solution constructed above satisfies
\begin{align*}
f_n(0)
=
\frac{(-i)^n}{\sqrt{\lambda_n}}.
\end{align*}
Since
$
b_n
=
\frac12\sqrt{\lambda_n\lambda_{n+1}},
$
we have
\begin{align*}
b_nf_{n+1}(0)
=
\frac12\sqrt{\lambda_n}\,(-i)^{n+1}.
\end{align*}
Hence \eqref{m-boundary-Jost-representation} yields
\begin{align*}
m_n(0+i0)
=
\frac{f_n(0)}{b_nf_{n+1}(0)}
=
\frac{2i}{\lambda_n}.
\end{align*}
In particular, since $\lambda_1=\frac12$,
\begin{align*}
m_1(0+i0)=4i.
\end{align*}
Thus the Jost denominator is nonzero at $E=0$; by continuity, it
remains nonzero in a sufficiently small neighborhood of zero.

Let $\mu$ be the spectral measure of $\mathcal{J}$ associated with
$e_1$. It follows that, in some neighborhood of zero,
\begin{align*}
d\mu(x)=w(x)\,dx,
\end{align*}
where $w$ is continuous at zero. By the Stieltjes inversion formula,
\begin{align}
\label{zero-density}
w(0)
=
\frac1\pi\operatorname{Im}m_1(0+i0)
=
\frac4\pi.
\end{align}

We also record directly that zero is not an eigenvalue of $\mathcal{J}$.
Suppose that $\mathcal{J}u=0$. The equation at $n=1$ gives $u_2=0$,
and the recurrence then yields
\begin{align*}
u_{2j}=0,
\qquad
u_{2j+1}
=
(-1)^j
\left(\frac{\lambda_1}{\lambda_{2j+1}}\right)^{1/2}u_1.
\end{align*}
Since $\lambda_{2j+1}\to1$, this sequence belongs to
$\ell^2(\mathbb N)$ only when $u_1=0$. Thus
\begin{align*}
\ker\mathcal{J}=\{0\},
\qquad
\ker\mathcal{G}=\{0\}.
\end{align*}
Because $\mathcal{G}^*=-\mathcal{G}$, it follows that
\begin{align}
\label{dense-range-A}
\overline{\operatorname{Ran}\mathcal{G}}
=
(\ker\mathcal{G}^*)^\perp
=
(\ker\mathcal{G})^\perp
=
\ell^2(\mathbb N).
\end{align}

\medskip
\noindent
\textbf{Step 3: Semigroup growth and viscous amplification.}

We now return to the integrated boundary forcing introduced above and
set
\begin{align*}
g_\tau^+
=
\int_0^\tau e^{-s\mathcal G}e_1\,ds,
\qquad
g_\tau^-
=
\int_0^\tau e^{s\mathcal G}e_1\,ds.
\end{align*}
The spectral theorem makes the preceding heuristic precise:
\begin{align}\label{g-spectral-representation}
\|g_\tau^\pm\|_{\ell^2}^2
=
\int_{\mathbb R}
\left|
\frac{e^{i\tau x}-1}{x}
\right|^2
d\mu(x).
\end{align}

We now extract the large-$\tau$ asymptotics. Choose $\delta>0$ such
that
\begin{align*}
d\mu(x)=w(x)\,dx
\qquad\text{for }|x|<\delta,
\end{align*}
with $w$ continuous there. Since
\begin{align*}
K_\tau(x)
=
\frac1\tau
\left|
\frac{e^{i\tau x}-1}{x}
\right|^2
\end{align*}
is an approximate identity of total mass $2\pi$, we have
\begin{align*}
\lim_{\tau\to\infty}
\frac1\tau
\int_{|x|<\delta}
\left|
\frac{e^{i\tau x}-1}{x}
\right|^2
w(x)\,dx
=
2\pi w(0).
\end{align*}
On the complement,
\begin{align*}
\frac1\tau
\int_{|x|\geq\delta}
\left|
\frac{e^{i\tau x}-1}{x}
\right|^2
d\mu(x)
\leq
\frac{4}{\tau\delta^2}\mu(\mathbb R)
\longrightarrow0.
\end{align*}
Therefore, by \eqref{zero-density},
\begin{align*}
\lim_{\tau\to\infty}
\frac1\tau\|g_\tau^\pm\|_{\ell^2}^2
=
2\pi w(0)
=
8,
\end{align*}
and hence
\begin{align}
\label{g-growth}
\|g_\tau^\pm\|_{\ell^2}^2
=
8\tau+o(\tau).
\end{align}

We also need the corresponding estimate in the dual of the weighted
norm in \eqref{X-norm}. First, we claim that
\begin{align}
\label{g-weak-convergence}
\frac{g_\tau^\pm}{\sqrt{\tau}}
\rightharpoonup0
\qquad\text{in }\ell^2(\mathbb N).
\end{align}
The family on the left is bounded by \eqref{g-growth}. If
$v=\mathcal{G}u\in\operatorname{Ran}\mathcal{G}$, then, using
$\mathcal{G}^*=-\mathcal{G}$,
\begin{align*}
\left|
\left\langle
\frac{g_\tau^+}{\sqrt{\tau}},v
\right\rangle
\right|
=
\frac1{\sqrt{\tau}}
\left|
\left\langle
-\mathcal{G}g_\tau^+,u
\right\rangle
\right|.
\end{align*}
Since
\begin{align*}
\mathcal{G}g_\tau^+
=
e_1-e^{-\tau\mathcal{G}}e_1,
\end{align*}
we obtain
\begin{align*}
\left|
\left\langle
\frac{g_\tau^+}{\sqrt{\tau}},v
\right\rangle
\right|
\leq
\frac{2\|u\|}{\sqrt{\tau}}
\longrightarrow0.
\end{align*}
The same argument applies to $g_\tau^-$. Since
$\operatorname{Ran}\mathcal{G}$ is dense by
\eqref{dense-range-A}, this proves \eqref{g-weak-convergence}.

Now
\begin{align*}
\mathcal{I}-\mathcal{D}_+
=
\operatorname{diag}
\left(\frac{1}{1+n^2}\right)_{n\geq1}
\end{align*}
is compact. Hence
\begin{align*}
(\mathcal{I}-\mathcal{D}_+)\frac{g_\tau^\pm}{\sqrt{\tau}}
\longrightarrow0
\qquad\text{strongly in }\ell^2(\mathbb N),
\end{align*}
and consequently
\begin{align*}
\frac1\tau
\left\langle
(\mathcal{I}-\mathcal{D}_+)g_\tau^\pm,g_\tau^\pm
\right\rangle
\longrightarrow0.
\end{align*}
Combining this with \eqref{g-growth}, we obtain
\begin{align}
\label{weighted-g-growth}
\|\mathcal{D}_+^{1/2}g_\tau^\pm\|_{\ell^2}^2
=
8\tau+o(\tau).
\end{align}

Integrating \eqref{zero-mode-forcing}, we find
\begin{align*}
d_0(\tau)
=
d_0(0)
+
\mathcal F_\tau(h^+(0),h^-(0)),
\end{align*}
where
\begin{align*}
\mathcal F_\tau(h^+,h^-)
=
\beta
\int_0^\tau
\left[
\left\langle e_1,e^{s\mathcal{G}}h^+\right\rangle
-
\left\langle e_1,e^{-s\mathcal{G}}h^-\right\rangle
\right]\,ds.
\end{align*}
Equivalently,
\begin{align*}
\mathcal F_\tau(h^+,h^-)
=
\beta
\left(
\langle g_\tau^+,h^+\rangle
-
\langle g_\tau^-,h^-\rangle
\right).
\end{align*}

For any $g\in\ell^2(\mathbb N)$, the weighted dual norm satisfies
\begin{align*}
\sup_{h\neq0}
\frac{|\langle g,h\rangle|^2}
{\|\mathcal{D}_+^{-1/2}h\|_{\ell^2}^2}
=
\|\mathcal{D}_+^{1/2}g\|_{\ell^2}^2.
\end{align*}
Therefore,
\begin{align*}
\|\mathcal F_\tau\|^2
=
\beta^2
\left(
\|\mathcal{D}_+^{1/2}g_\tau^+\|_{\ell^2}^2
+
\|\mathcal{D}_+^{1/2}g_\tau^-\|_{\ell^2}^2
\right).
\end{align*}
Since $\beta^2=1/8$, \eqref{weighted-g-growth} gives
\begin{align}\label{F-growth}
\|\mathcal F_\tau\|^2
=
2\tau+o(\tau),
\qquad
\|\mathcal F_\tau\|
=
\sqrt{2\tau}\,(1+o(1)).
\end{align}

Under the isometric identification with $\mathcal Y$, the inviscid
semigroup has the upper triangular form
\begin{align*}
e^{\tau \mathcal{M}_1}
=
\begin{pmatrix}
1&\mathcal F_\tau\\
0&\mathcal V_\tau
\end{pmatrix},
\qquad
\mathcal V_\tau
=
e^{\tau\mathcal{G}}\oplus e^{-\tau\mathcal{G}}.
\end{align*}
Since $\mathcal{G}$ is skew-adjoint and
$\mathcal{D}_+^{\pm1/2}$ are bounded, the family
$\{\mathcal V_\tau\}_{\tau\in\mathbb R}$ is uniformly bounded on the
weighted space. Consequently,
\begin{align*}
\|\mathcal F_\tau\|
\leq
\|e^{\tau \mathcal{M}_1}\|
\leq
\|\mathcal F_\tau\|+C.
\end{align*}
Together with \eqref{F-growth}, this proves
\begin{align}\label{critical-inviscid-asymptotics-proof}
\|e^{\tau \mathcal{M}_1}\|
=
\sqrt{2\tau}\,(1+o(1)),
\qquad
\tau\to\infty.
\end{align}
By the symmetry argument at the beginning of the proof, the same
asymptotic formula holds for $k=-1$.

We next complement the preceding asymptotic formula with a global
inviscid bound, which will be needed to optimize the viscous
amplification uniformly in time.
The continuity and boundedness of $w$ near zero also imply
\begin{align*}
\|g_\tau^\pm\|_{\ell^2}^2
\leq
C(1+\tau),
\qquad \tau\geq0.
\end{align*}
Indeed, this follows from \eqref{g-spectral-representation} and
\begin{align*}
\left|
\frac{e^{i\tau x}-1}{x}
\right|^2
\leq
\min\left\{\tau^2,\frac4{x^2}\right\}.
\end{align*}
Hence
\begin{align}\label{critical-inviscid-global}
\|e^{\tau \mathcal{M}_k}\|
\leq
C(1+\sqrt{\tau}),
\qquad
\tau\geq0,
\qquad |k|=1.
\end{align}

We now return to the viscous evolution. The exact factorization gives
\begin{align*}
e^{t\mathcal A_{k,\nu}}
=
e^{-\nu t}e^{at\mathcal{M}_k},
\qquad |k|=1.
\end{align*}
Therefore,
\begin{align*}
\|e^{t\mathcal A_{k,\nu}}\|
\leq
C(1+\sqrt{at})e^{-\nu t},
\qquad t\geq0,
\end{align*}
which proves \eqref{critical-global-bound}.

It remains to justify the sharp maximal amplification. Set
\begin{align*}
\varepsilon=\frac{\nu}{a},
\qquad
N(\tau)=\|e^{\tau \mathcal{M}_k}\|.
\end{align*}
Then
\begin{align*}
\sup_{t\geq0}\|e^{t\mathcal A_{k,\nu}}\|
=
\sup_{\tau\geq0}e^{-\varepsilon\tau}N(\tau).
\end{align*}
After setting $y=\varepsilon\tau$, we obtain
\begin{align*}
\sqrt{\varepsilon}
\sup_{\tau\geq0}e^{-\varepsilon\tau}N(\tau)
=
\sup_{y\geq0}
e^{-y}\sqrt{\varepsilon}\,
N\left(\frac{y}{\varepsilon}\right).
\end{align*}
For each fixed $y>0$, \eqref{critical-inviscid-asymptotics-proof}
gives
\begin{align*}
\sqrt{\varepsilon}\,
N\left(\frac{y}{\varepsilon}\right)
\longrightarrow
\sqrt{2y}.
\end{align*}
The convergence is uniform on compact subsets of $(0,\infty)$.
Moreover, the global bound above yields
\begin{align*}
e^{-y}\sqrt{\varepsilon}\,
N\left(\frac{y}{\varepsilon}\right)
\leq
C e^{-y}(\sqrt{\varepsilon}+\sqrt y).
\end{align*}
This estimate controls uniformly the regions near $y=0$ and
$y=\infty$. Therefore,
\begin{align*}
\lim_{\varepsilon\downarrow0}
\sqrt{\varepsilon}
\sup_{\tau\geq0}e^{-\varepsilon\tau}N(\tau)
=
\max_{y\geq0}e^{-y}\sqrt{2y}
=
e^{-1/2},
\end{align*}
where the maximum is attained at $y=1/2$. Consequently,
\begin{align*}
\sup_{t\geq0}\|e^{t\mathcal A_{k,\nu}}\|
=
\left(\frac{a}{e\nu}\right)^{1/2}(1+o(1)),
\qquad
\frac{a}{\nu}\to\infty.
\end{align*}
 At
$
t_*=\frac{1}{2\nu},$
$
at_*=\frac{a}{2\nu}\to\infty,
$
the inviscid asymptotic \eqref{critical-inviscid-asymptotics}
gives
\begin{align*}
\begin{aligned}
\left\|e^{t_*\mathcal A_{k,\nu}}\right\|
&=
e^{-\nu t_*}
\left\|e^{at_*\mathcal{M}_k}\right\| \\
&=
e^{-1/2}
\left\|e^{\frac{a}{2\nu}\mathcal{M}_k}\right\| \\
&=
e^{-1/2}
\sqrt{\frac{a}{\nu}}\,(1+o(1)) \\
&=
\left(\frac{a}{e\nu}\right)^{1/2}(1+o(1)).
\end{aligned}
\end{align*}
This proves \eqref{critical-maximizing-time} and shows that
$t_*=(2\nu)^{-1}$ is asymptotically maximizing.

 The global estimate \eqref{critical-global-bound} gives
\begin{align*}
\begin{aligned}
\sup_{t\geq0}
\left\|e^{t\mathcal A_{k,\nu}}\right\|
&\leq
C\sup_{t\geq0}
(1+\sqrt{at})e^{-\nu t}\leq
C\left(1+\sqrt{\frac a\nu}\right).
\end{aligned}
\end{align*}
For $a/\nu$ sufficiently large, the matching lower bound follows
from \eqref{critical-viscous-asymptotics}. For $a/\nu$ bounded,
we use
\begin{align*}
\sup_{t\geq0}
\left\|e^{t\mathcal A_{k,\nu}}\right\|
\geq
\|\mathcal{I}\|=1.
\end{align*}
After adjusting the absolute constants, we obtain
\begin{align*}
\sup_{t\geq0}
\left\|e^{t\mathcal A_{k,\nu}}\right\|
\asymp
1+\sqrt{\frac a\nu}.
\end{align*}
This proves \eqref{critical-two-sided}.

Finally, \eqref{half-line-evolution} and the skew-adjointness of
$\mathcal{G}$ imply
\begin{align*}
\sum_{n\neq0}\lambda_n|d_n(\tau)|^2
=
\sum_{n\neq0}\lambda_n|d_n(0)|^2.
\end{align*}
Since $d(at)=e^{\nu t}c(t)$, this yields
\begin{align*}
e^{2\nu t}
\sum_{n\neq0}\lambda_n|c_n(t)|^2
=
\sum_{n\neq0}\lambda_n|c_n(0)|^2,
\end{align*}
which is the conservation law
\eqref{critical-conservation}.

Moreover, $\mathcal{M}_ke_0=0$, and hence
\begin{align*}
e^{t\mathcal A_{k,\nu}}e_0=e^{-\nu t}e_0.
\end{align*}
Combining this lower bound with
\begin{align*}
\|e^{t\mathcal A_{k,\nu}}\|
\leq
C(1+\sqrt{at})e^{-\nu t}
\end{align*}
gives
\begin{align*}
\lim_{t\to\infty}
\frac1t
\log
\|e^{t\mathcal A_{k,\nu}}\|
=
-\nu.
\end{align*}
This completes the proof.
\end{proof}

\begin{remark}[Transient growth at the critical wavenumbers]
The critical growth is generated by the coupling of the one-dimensional
zero mode to the uniformly bounded positive and negative half-line
subsystems. Thus the $\sqrt{t}$-growth is algebraic rather than
spectral: it produces substantial transient amplification, but does not
alter the spectral bound or the asymptotic exponential decay rate. Indeed,
\begin{align*}
\left.\frac{d}{dt}d_0(t)\right|_{t=0}
=
-\frac{ka}{4}\bigl(d_{-1}(0)-d_1(0)\bigr),
\end{align*}
so the neutral component is directly forced by the neighboring
nonzero modes.

\end{remark}

\subsection{The undamped kernel $k=0$: proof of Theorem \ref{thm:k0}}
\begin{proof}[Proof of Theorem~\ref{thm:k0}]
Let $\mathbb P_0$ denote the projection onto the horizontally
independent component:
\begin{align*}
(\mathbb P_0\omega)(x_2)
=
\frac{1}{L}\int_{\mathbb T_L}
\omega(x_1,x_2)\,dx_1
=
\widehat{\omega}_0(x_2).
\end{align*}
Since the coefficients in the linearized vorticity equation
\eqref{linvort} are independent of $x_1$, the horizontal Fourier
sectors are invariant. On the mode $k=0$, both
$\partial_1$ and $\partial_1^2$ vanish. Therefore,
\begin{align*}
\partial_t\widehat{\omega}_0=0,
\qquad
\widehat{\omega}_0(t)=\widehat{\omega}_0(0),
\qquad t\geq0.
\end{align*}

Recall that
\begin{align*}
\mathcal K_0
:=
\left\{
\omega(x_1,x_2)=g(x_2):
g\in L^2(\mathbb T_{2\pi}),\
\int_0^{2\pi}g(x_2)\,dx_2=0
\right\}.
\end{align*}
The zero-mean condition ensures that $\Delta^{-1}$ is well-defined
on this mode. The preceding identity consequently gives
\begin{align*}
\Lan g=0,
\qquad
e^{t\Lan }g=g,
\qquad
g\in\mathcal K_0,\quad t\geq0.
\end{align*}
Thus
\begin{align*}
\Lan \big|_{\mathcal K_0}=0,
\qquad
e^{t\Lan }\big|_{\mathcal K_0}=\mathcal{I}.
\end{align*}

We next identify this kernel at the velocity level. Given
$g\in\mathcal K_0$, there exists a periodic function
$f\in H^1(\mathbb T_{2\pi})$ satisfying
\begin{align*}
-f'=g.
\end{align*}
After fixing the mean of the horizontal velocity by imposing
\begin{align*}
\int_0^{2\pi}f(x_2)\,dx_2=0,
\end{align*}
the function $f$ is uniquely determined. The associated velocity
perturbation is
\begin{align*}
u=(f(x_2),0),
\end{align*}
which is divergence-free and has vorticity
\begin{align*}
\partial_1u_2-\partial_2u_1=-f'=g.
\end{align*}
Conversely, every mean-zero shear perturbation
\begin{align*}
u=(f(x_2),0),
\qquad
f\in H^1(\mathbb T_{2\pi}),
\qquad
\int_0^{2\pi}f(x_2)\,dx_2=0,
\end{align*}
has vorticity in $\mathcal K_0$. Hence, under the
zero-mean Biot--Savart correspondence,
\begin{align*}
\mathcal K_0
\longleftrightarrow
\mathcal S_0.
\end{align*}

Since $\mathcal K_0$ is infinite-dimensional and the semigroup acts
as the identity on it, no nonzero element of $\mathcal K_0$ decays.
Consequently, the linearized semigroup is not strongly stable on any
phase space containing $\mathcal K_0$.

Finally, the invariance of the horizontal Fourier sectors implies that
$\mathbb P_0$ commutes with the linearized semigroup:
\begin{align*}
\mathbb P_0e^{t\Lan }
=
e^{t\Lan }\mathbb P_0.
\end{align*}
Therefore,
\begin{align*}
\ker\mathbb P_0
=
\{\omega:\widehat{\omega}_0=0\}
\end{align*}
is an invariant complementary subspace. All decay and stability
statements for the nonzero horizontal modes may consequently be
restricted to $\ker\mathbb P_0$.

If the horizontal mean velocity is not prescribed, the additive
constant in $f$ is not fixed. In that case, the velocity kernel also
contains the constant horizontal mode $\operatorname{span}\{(1,0)\}$;
equivalently, one may allow arbitrary periodic $f$ in the definition
of the shear space.
\end{proof}

\begin{remark}[Manifold of shear equilibria]
The kernel in Theorem~\ref{thm:k0} is tangent to the manifold
\begin{align*}
\mathcal E
=
\Big\{(f(x_2),0): f\ \text{periodic}\Big\}
\end{align*}
of shear equilibria. Thus the natural nonlinear stability problem is
stability modulo $\mathcal E$, rather than asymptotic stability of the
individual Kolmogorov profile.
\end{remark}

\subsection{{The exact spectral shift and the unstable range: proof of Theorem \ref{thm:B}}}
\begin{proof}[Proof of Theorem~\ref{thm:B}]
\textup{(i)}
Since $\mathcal{M}_k\in\mathcal B(\ell^2(\mathbb Z))$, the operator
\begin{align*}
\mathcal A_{k,\nu}=a\mathcal{M}_k-\nu k^2\mathcal{I}
\end{align*}
is bounded on $\ell^2(\mathbb Z)$. By the translation and scaling
properties of the spectrum,
\begin{align*}
\spec(\mathcal A_{k,\nu})
=
a\,\spec(\mathcal{M}_k)-\nu k^2.
\end{align*}
In particular,
\begin{align*}
s(\mathcal A_{k,\nu})
:=
\sup\{\Rea z:z\in\spec(\mathcal A_{k,\nu})\}
=
a\Lambda(k)-\nu k^2.
\end{align*}

By the spectral mapping theorem,
\begin{align*}
\spec\bigl(e^{t\mathcal A_{k,\nu}}\bigr)
=
e^{t\spec(\mathcal A_{k,\nu})},
\qquad t>0,
\end{align*}
and hence
\begin{align*}
r\bigl(e^{t\mathcal A_{k,\nu}}\bigr)
=
e^{t s(\mathcal A_{k,\nu})},
\end{align*}
where $r(\cdot)$ denotes the spectral radius.

Applying the spectral-radius formula to
$T=e^{\mathcal A_{k,\nu}}$, we obtain
\begin{align*}
\lim_{m\to\infty}
\frac1m\log\|e^{m\mathcal{A}_{k,\nu}}\|
=
\log r(e^{\mathcal A_{k,\nu}})
=
s(\mathcal A_{k,\nu}).
\end{align*}
It remains to pass from integer times to arbitrary times. Let
$t=m+s$, where $m\in\mathbb N$ and $0\leq s<1$. Since the group
property yields
$
e^{t\mathcal A_{k,\nu}}
=
e^{s\mathcal A_{k,\nu}}
e^{m\mathcal A_{k,\nu}},
$
and hence
\begin{align*}
\bigl\|e^{m\mathcal A_{k,\nu}}\bigr\|
\leq
\bigl\|e^{-s\mathcal A_{k,\nu}}\bigr\|
\bigl\|e^{t\mathcal A_{k,\nu}}\bigr\|.
\end{align*}

Because $\mathcal A_{k,\nu}$ is bounded,
\begin{align*}
C_{k,\nu}
:=
\sup_{0\leq s\leq 1}
\max\Big\{
\bigl\|e^{s\mathcal A_{k,\nu}}\bigr\|,
\bigl\|e^{-s\mathcal A_{k,\nu}}\bigr\|
\Big\}
<\infty.
\end{align*}
Thus
\begin{align*}
C_{k,\nu}^{-1}
\bigl\|e^{m\mathcal A_{k,\nu}}\bigr\|
\leq
\bigl\|e^{t\mathcal A_{k,\nu}}\bigr\|
\leq
C_{k,\nu}
\bigl\|e^{m\mathcal A_{k,\nu}}\bigr\|.
\end{align*}
Since $m/t\to1$ as $t\to\infty$, the integer-time limit implies
\begin{align*}
\lim_{t\to\infty}
\frac{1}{t}
\log
\bigl\|e^{t\mathcal A_{k,\nu}}\bigr\|
=
s(\mathcal A_{k,\nu})
=
a\Lambda(k)-\nu k^2.
\end{align*}

\medskip
\noindent
\textup{(ii)}
Setting $\nu=0$ and $a=1$ in \eqref{tridiag} shows that $\mathcal{M}_k$
is the generator of the mode-$k$ linearization of the Euler equations
about \eqref{Kolmo} with unit amplitude. Thus
\begin{align*}
\Lambda(k)
=
\max\{\Rea z:z\in\spec(\mathcal{M}_k)\}
\end{align*}
is the inviscid spectral growth rate per unit amplitude.

Since
$
\lambda_n^{(-k)}=\lambda_n^{(k)},
$
the definition of $\mathcal{M}_k$ gives
\begin{align*}
\mathcal{M}_{-k}=-\mathcal{M}_k.
\end{align*}
Define the unitary involution $\mathcal{R}$ on $\ell^2(\mathbb Z)$ by
\begin{align*}
(\mathcal{R}c)_n:=c_{-n}.
\end{align*}
Using $\lambda_{-n}^{(k)}=\lambda_n^{(k)}$, we compute
\begin{align*}
\begin{aligned}
(\mathcal{R}\mathcal{M}_k\mathcal{R}c)_n
&=(\mathcal{M}_k\mathcal{R}c)_{-n} \\
&=-\frac{k}{2}
\left[
\lambda_{-n-1}^{(k)}(\mathcal{R}c)_{-n-1}
-
\lambda_{-n+1}^{(k)}(\mathcal{R}c)_{-n+1}
\right] \\
&=-\frac{k}{2}
\left[
\lambda_{n+1}^{(k)}c_{n+1}
-
\lambda_{n-1}^{(k)}c_{n-1}
\right] \\
&=-(\mathcal{M}_kc)_n.
\end{aligned}
\end{align*}
Thus
\begin{align*}
\mathcal{R}\mathcal{M}_k\mathcal{R}^{-1}=-\mathcal{M}_k.
\end{align*}
Consequently,
\begin{align*}
\spec(\mathcal{M}_k)=\spec(-\mathcal{M}_k)=-\spec(\mathcal{M}_k).
\end{align*}
Combining this symmetry with $\mathcal{M}_{-k}=-\mathcal{M}_k$, we find
\begin{align*}
\spec(\mathcal{M}_{-k})
=
\spec(-\mathcal{M}_k)
=
\spec(\mathcal{M}_k),
\end{align*}
and therefore
\begin{align*}
\Lambda(-k)=\Lambda(k).
\end{align*}

For $|k|>1$, Lemma~\ref{lem:skew} implies that $\mathcal{M}_k$ is similar,
through a bounded and boundedly invertible transformation, to a bounded
skew-adjoint operator. Hence
\begin{align*}
\spec(\mathcal{M}_k)\subset i\mathbb R,
\qquad
\Lambda(k)=0.
\end{align*}

Let now $|k|=1$. {The inviscid bound \eqref{critical-inviscid-global}, established in the proof of Theorem~\ref{thm:k1} without using $\nu>0$, yields}
\begin{align*}
\|e^{t\mathcal{M}_k}\|_{\ell^2\to\ell^2}
\leq C(1+\sqrt t),
\qquad t\geq0.
\end{align*}
For $z\in\spec(\mathcal{M}_k)$, the spectral mapping theorem gives
$
e^{tz}\in\spec(e^{t\mathcal{M}_k}),
$
and therefore
\begin{align*}
e^{t\Rea z}
\leq
r(e^{t\mathcal{M}_k})
\leq
\|e^{t\mathcal{M}_k}\|
\leq
C(1+\sqrt t).
\end{align*}
Taking logarithms, dividing by $t$, and letting $t\to\infty$, we obtain
$
\Rea z\leq0.
$
Thus
\begin{align*}
\spec(\mathcal{M}_k)
\subset
\{z\in\mathbb C:\Rea z\leq0\}.
\end{align*}
Since $\spec(\mathcal{M}_k)=-\spec(\mathcal{M}_k)$, the same conclusion applied to $-z$
gives $\Rea z\geq0$. Hence
\begin{align*}
\spec(\mathcal{M}_k)\subset i\mathbb R,
\qquad
\Lambda(\pm1)=0.
\end{align*}

\medskip
\noindent
\textup{(iii)}
Part~\textup{(i)} gives
\begin{align*}
s(\mathcal A_{k,\nu})
=
a\Lambda(k)-\nu k^2.
\end{align*}
Hence the $k$-th mode is spectrally unstable if and only if
\begin{align*}
a\Lambda(k)>\nu k^2.
\end{align*}
Whenever this condition holds, necessarily $\Lambda(k)>0$.
Corollary~\ref{cor:unstable-eigenmode} then yields an eigenvalue
$z_k\in\sigma_{\mathrm p}(\mathcal{M}_k)$ satisfying
\begin{align*}
\operatorname{Re}z_k=\Lambda(k).
\end{align*}
Consequently,
\begin{align*}
az_k-\nu k^2
\in
\sigma_{\mathrm p}(\mathcal A_{k,\nu})
\end{align*}
lies in the open right half-plane, and the corresponding eigenvector
generates an exponentially growing normal mode. Its growth exponent is
$
a\Lambda(k)-\nu k^2,
$
which, by part~\textup{(i)}, also equals the asymptotic operator-norm
growth exponent.

\medskip
\noindent
\textup{(iv)}
For $0<|k|<1$, the classical inviscid instability theory for the
Kolmogorov flow yields
$
\Lambda(k)>0,
$
see {\cite{MeshalkinSinai,FriedlanderStraussVishik,FriedlanderHoward,BelenkayaFriedlanderYudovich,Lin}}. Therefore
\begin{align*}
a>\frac{\nu k^2}{\Lambda(k)}
\end{align*}
is equivalent to
$
a\Lambda(k)-\nu k^2>0.
$
The instability criterion and the exact growth exponent now follow from
part \textup{(iii)}. In particular, for fixed $k$ and $\nu$, the
growth exponent depends linearly on $a$.
\end{proof}
\begin{remark}[Continued-fraction characterization]
For $0<|k|<1$, the discrete eigenvalues of $\mathcal{M}_k$ may also be
characterized by the classical continued-fraction method. Indeed, if
\begin{align*}
\mathcal{M}_kc=\sigma c,
\qquad
c\in\ell^2(\mathbb Z),
\end{align*}
and $r_n:=\lambda_n c_n$, then
\begin{align*}
\frac{\sigma}{\lambda_n}r_n
=
-\frac{k}{2}\bigl(r_{n-1}-r_{n+1}\bigr),
\qquad n\in\mathbb Z.
\end{align*}
For fixed $\sigma$ outside the essential spectrum, the
$\ell^2$-solutions on the two half-lines are encoded by convergent
continued fractions. Matching these solutions across $n=0$ yields
the corresponding secular equation for the discrete eigenvalues; see
{\cite{MeshalkinSinai,BelenkayaFriedlanderYudovich,LatushkinLiStanislavova}}. This characterization is not needed for the
spectral-shift and instability results above.
\end{remark}

\begin{remark}[Sharpness in $a$ and comparison with the classical threshold]
In the classical, fully dissipative (forced) problem the neutral curve is determined by a genuine competition between the $n$-dependent damping $\nu(k^2+n^2)$ and the advection, and the threshold is known only through the continued-fraction analysis \cite{MeshalkinSinai,Yudovich}. In the anisotropic problem the exact spectral shift of Theorem \ref{thm:B}(i) collapses this competition: the neutral curve is exactly $\{a\Lambda(k)=\nu k^2\}$, the bifurcation at the threshold occurs, at the level of the mode-$k$ spectrum, by a rigid horizontal translation of the inviscid spectrum, and the large-$a$ asymptotics of the growth rate is exactly linear. We regard this exact solvability as the main structural payoff of the anisotropic dissipation.
\end{remark}
\subsection{{Proof of Corollary \ref{cor:L}}}\label{sec:proofcor}
\begin{proof}[{Proof of Corollary~\ref{cor:L}}]
Under the horizontal Fourier decomposition, the invariant space
$\ker\mathbb P_0$ is the orthogonal direct sum of the sectors
$k\neq0$, and the linearized semigroup acts diagonally with respect
to this decomposition. Consequently,
\begin{align*}
\left\|
e^{t\Lan }
\big|_{\ker\mathbb P_0}
\right\|
=
\sup_{k\in(2\pi/L)\mathbb Z\setminus\{0\}}
\left\|e^{t\mathcal A_{k,\nu}}\right\|.
\end{align*}

If $L<2\pi$, then every nonzero admissible wavenumber satisfies
$|k|\geq k_{\min}>1$. Theorem~\ref{thm:A} gives
\begin{align*}
\left\|e^{t\mathcal A_{k,\nu}}\right\|
\leq
\left(1-\frac1{k^2}\right)^{-1/2}
e^{-\nu k^2t}
\leq
C_L e^{-\nu k_{\min}^2t}.
\end{align*}
Taking the supremum over $k\neq0$ proves the asserted upper bound.
The lower estimate in \eqref{twosided}, applied to
$k=\pm k_{\min}$, yields the matching exponential rate.

If $L=2\pi$, the critical sectors $k=\pm1$ satisfy the estimate in
Theorem~\ref{thm:k1}. For $|k|\geq2$, Theorem~\ref{thm:A} gives
\begin{align*}
\left\|e^{t\mathcal A_{k,\nu}}\right\|
\leq
\left(1-\frac14\right)^{-1/2}e^{-4\nu t}.
\end{align*}
Taking the supremum over all $k\neq0$ proves the stated bound.
The exact exponential rate follows from the critical-mode lower
bound
\begin{align*}
\left\|e^{t\mathcal A_{\pm1,\nu}}\right\|
\geq e^{-\nu t}.
\end{align*}

Finally, suppose that $L>2\pi$. The set $\mathcal K_L$ is finite
and nonempty. By Theorem~\ref{thm:B} and
Corollary~\ref{cor:unstable-eigenmode}, a long-wave mode $k$ is
exponentially unstable precisely when
\begin{align*}
a\Lambda(k)>\nu k^2.
\end{align*}
All sectors with $|k|\geq1$ have nonpositive asymptotic growth
exponents. Hence the full restricted semigroup is exponentially
unstable exactly when the preceding inequality holds for some
$k\in\mathcal K_L$, equivalently when
$a>a_{\mathrm c}(L,\nu)$. In that case, its growth exponent is the
maximum of the finitely many unstable long-wave exponents, as claimed.
\end{proof}

\section{Open problems: the nonlinear theory}

We close by formulating the nonlinear questions that the linear theory above suggests; they will be the object of a subsequent work.

\subsection{Nonlinear stability modulo shears for $L<2\pi$}
Let $L<2\pi$, so that every nonzero mode enjoys the sharp decay \eqref{twosided}. Theorem \ref{thm:k0} indicates that the correct conjecture is: for initial data $U_0$ with $\|U_0-U^{(0)}\|_{H^s}\le\eps_0$ small (in a suitable anisotropic space, uniformly in $a$), the solution of \eqref{ans} exists globally and converges, as $t\to\infty$, to a shear flow $(f_\infty(x_2),0)$ with $\|f_\infty-a\sin x_2\|$ small; moreover the non-shear part decays at the linear rate $e^{-\nu k_{\min}^2t}$. The expected difficulties are (a) the absence of dissipation on the shear component, which must be controlled by the structure of the nonlinearity (the shear component is driven only quadratically by decaying modes, so its total drift is finite for small data); and (b) the closure of the quadratic estimates given that vertical derivatives are never regularized --- here the divergence-free structure, which converts vertical structure of the velocity into horizontal derivatives via $\p_2u_2=-\p_1u_1$, should play the role it plays in the partially dissipative literature. The two-sided linear bound \eqref{twosided} suggests that, unlike in the enhanced-dissipation setting, the smallness $\eps_0$ can be taken uniform in $a$ for $L<2\pi$: at the linear level, $a$ simply does not appear. Whether the nonlinear interaction of the neutral shear modes with the decaying modes destroys this uniformity for large $a$ (through the $a$-dependent transient of the marginal case when $L=2\pi$, or through the quadratic feedback onto the kernel) is, in our view, the most interesting open question in the stable regime.

\subsection{Nonlinear dynamics near the threshold for $L>2\pi$}
Fix an admissible wavenumber $0<|k|<1$, and set
\begin{align*}
a_{\mathrm c}(k,\nu)
:=
\frac{\nu k^2}{\Lambda(k)}.
\end{align*}
Let $z_k\in\sigma_{\mathrm p}(\mathcal{M}_k)$ satisfy
$
\operatorname{Re}z_k=\Lambda(k).
$
The corresponding eigenvalue of the viscous operator is
\begin{align*}
az_k-\nu k^2,
\end{align*}
whose real part reaches zero at $a=a_{\mathrm c}(k,\nu)$. Moreover,
\begin{align*}
\frac{d}{da}
\operatorname{Re}\bigl(az_k-\nu k^2\bigr)
=
\Lambda(k)>0,
\end{align*}
so the crossing is transversal at the level of this eigenvalue.

If $z_k$ is simple and the appropriate nonlinear nondegeneracy
conditions hold, one may expect a local bifurcation of secondary
coherent structures from the Kolmogorov flow, in analogy with the
classical viscous theory of Yudovich {\cite{Yudovich}; see also \cite{OkamotoShoji,MatsudaMiyatake} for bifurcation analyses of Kolmogorov flows}. The nature of the bifurcating
solutions may depend on whether the critical eigenvalue is real or
nonreal. Establishing such a bifurcation and describing the nonlinear
dynamics for $a$ near $a_{\mathrm c}(k,\nu)$ remain open problems.

\subsection{Transverse dissipation}
Finally, the companion model with $b=(0,1)$ {---} vertical dissipation with a horizontal shear --- is expected to behave in the opposite way: the shear transfer feeds the dissipated direction, and enhanced dissipation with an $a$-dependent time scale should reappear. {This expectation is consistent with recent work on shear flows with vertical dissipation, in which the transfer of nonzero horizontal modes to large vertical frequencies does produce enhanced dissipation; see \cite{DengWuZhang} for the Boussinesq system and \cite{LiangWuZhai} for the MHD system near Couette flow.} A sharp comparison of the two anisotropic models, for the same background \eqref{Kolmo}, would isolate the role of the relative orientation of shear and dissipation in the stability of parallel flows.

\vskip .2in
\subsection*{Acknowledgments}

W. Yang was partially supported by the National Natural Science Foundation of China (No. 12561036), the Natural Science Foundation of Ningxia (No. 2023AAC02044), Ningxia Youth Top-notch Talents Cultivation Program and Ningxia Foreign Talents Introduction Program(2026). J.~Wu was partially supported by the
National Science Foundation of the United States (Grants DMS~2104682 and
DMS~2309748). X.~Zhai was partially supported by the Guangdong Provincial
Natural Science Foundation (Grants 2024A1515030115).

 \vskip .1in
\subsection*{Data Availability Statement} Data sharing is not applicable to this article as no
data sets were generated or analysed during the current study.

\vskip .1in

\subsection*{Conflict of Interest} The authors declare that they have no conflict of interest. The
authors also declare that this manuscript has not been previously published, and
will not be submitted elsewhere before your decision.

\vskip .1in
\subsection*{Declaration of generative AI and AI-assisted technologies in the manuscript preparation process}
 During the preparation of this work, the authors used Gemini 3.1-Pro to improve the English language. After using this
tool, the authors reviewed and edited the content as needed and take full responsibility for the final
version of the manuscript.

\end{document}